\documentclass[12pt]{amsart}
\usepackage{latexsym, url}
\usepackage{amsthm}
\usepackage{amsmath}
\usepackage{amsfonts}
\usepackage{amssymb}
\usepackage[dvips]{graphicx}
\usepackage{xypic}
\usepackage{xcolor}
\usepackage{tikz}
\usetikzlibrary{arrows}
\usetikzlibrary{calc}
\usepackage[bookmarks=false]{hyperref}
\input xy
\xyoption{all}

\newtheorem{theo}{Theorem}[section]
\newtheorem{lemma}[theo]{Lemma}
\newtheorem{assume}[theo]{Assumption}
\newtheorem{propo}[theo]{Proposition}
\newtheorem{coro}[theo]{Corollary}
\newtheorem*{theo:main}{Theorem~\ref{char}}
\theoremstyle{definition}
\newtheorem{defi}[theo]{Definition}
\newtheorem{nota}[theo]{Notation}
\newtheorem{rem}[theo]{Remark}

\newtheorem{exam}[theo]{Example}
\newtheorem{exams}[theo]{Examples}

\newcommand\Mod{\operatorname{\bf Mod}}

\newcommand\Set{\operatorname{\bf Set}}
\newcommand\Met{\operatorname{\bf Met}}
\newcommand\Cat{\operatorname{\bf Cat}}

\newcommand\Str{\operatorname{\bf Str}}
\newcommand\CMet{\operatorname{\bf CMet}}
\newcommand\Ab{\operatorname{\bf Ab}}
\newcommand\Ban{\operatorname{\bf Ban}}
\newcommand\Pos{\operatorname{\bf Pos}}

\newcommand\CPO{\operatorname{\bf CPO}}

\newcommand\colim{\operatorname{colim}}

\newcommand\eps{\varepsilon}

\newcommand\ca{\mathcal {A}}

\newcommand\cg{\mathcal {G}}
\newcommand\ch{\mathcal {H}}
\newcommand\ci{\mathcal {I}}
\newcommand\ce{\mathcal {E}}

\newcommand\ck{\mathcal {K}}
\newcommand\cl{\mathcal {L}}
\newcommand\cm{\mathcal {M}}

\newcommand\cp{\mathcal {P}}

\newcommand\cv{\mathcal {V}}

\newcommand{\RR}{{\mathbb R}}
\newcommand{\LL}{{\mathbb L}}

\newcommand{\TT}{{\mathbb T}}

\newcommand{\tx}{\textnormal}
\newcommand{\bo}{\mathbf}
\date{October 15, 2026}
 
\begin{document}
\title[Enriched positive logic]
{Enriched positive logic}
\author[J. Rosick\'{y} and G. Tendas]
{J. Rosick\'{y} and G. Tendas}
\thanks{We thank the anonymous referee for helpful comments that helped us improve the results of the paper. The first author acknowledges the support of the Grant Agency of the Czech Republic under the grant 22-02964S. The second author acknowledges the support of the EPSRC postdoctoral fellowship EP/X027139/1.} 
\address{
\newline J. Rosick\'{y}\newline
Department of Mathematics and Statistics\newline
Masaryk University, Faculty of Sciences\newline
Kotl\'{a}\v{r}sk\'{a} 2, 611 37 Brno, Czech Republic\newline
\textnormal{rosicky@math.muni.cz}\newline
\newline G. Tendas\newline
Department of Mathematics\newline
University of Manchester, Faculty of Science and Engineering\newline
Alan Turing Building, M13 9PL Manchester, UK\newline
\textnormal{giacomo.tendas@manchester.ac.uk}
}
 
\begin{abstract}
Building on our previous work on enriched regular logic, we introduce an
enriched version of positive logic and relate it to enriched cone-injectivity classes and enriched accessible categories. To do this, we need a factorization system on the base of enrichment in order to interpret existential quantification and disjunctions. We will also show how to treat unique existence in enriched logic, and how to relate it to local presentability.
\end{abstract} 
\keywords{}
\subjclass{}

\maketitle

\tableofcontents

\section{Introduction}
Since the early developments of the theory, the relationship between accessible categories and logic was well understood~\cite{MP,AR}. Indeed, it was soon established that categories of models of specific fragments of logic, correspond to certain kinds of accessible categories. For instance, the categories of models of {\em cartesian} (or limit) theories are the same as complete accessible (that is, locally presentable) categories, while categories of models of (infinitary) {\em regular} theories are the same as accessible categories with products. In full generality, categories of models of {\em basic} theories are the same as accessible categories, with no limits required.

Alternatively, the same classes of categories can be described through injectivity and orthogonality conditions. Locally presentable categories are orthogonality classes in presheaf categories, accessible categories with products arise as injectivity classes, and general accessible categories as cone-injectivity classes. Such descriptions, although not being intrinsic to the category taken into account (since they are defined within a larger category), are quite useful to produce examples. 

When moving to the enriched context things become more technical, and only some of the correspondences above, involving orthogonality and injectivity, have been carried out. Indeed, locally presentable enriched categories are characterized as orthogonality classes in~\cite{K,Kel82}, and the relationship between enriched accessibility and injectivity has been first studied in~\cite{LR12}. 

From the point of view of logic no comparison was made for a long time, since a theory of {\em enriched logic} was missing from the literature. To fill this void, in~\cite{RTe} we introduced enriched equational theories with recursively generated terms, and then in~\cite{RTe2} we added relation symbols, conjunctions, and existential quantification to produce the regular fragment of enriched logic. In this last work, we characterized enriched injectivity classes as categories of models of enriched regular theories, generalizing the ordinary results. 

Here we introduce {\em positive logic} and {\em basic theories}, in the enriched framework, and characterize their categories of models in terms of enriched cone-injectivity classes (that we shall introduce) and of enriched accessible categories.

More in detail, in \cite{RTe2}, enriched regular logic is built as the logic
of positive-primitive formulas; that is, formulas built by conjunctions
and existential quantifications (both allowed to be infinitary) from atomic formulas. In order
to interpret the existential quantification, we needed an enriched factorization system $(\ce,\cm)$ on the base $\cv$ of enrichment. 
In the ordinary ($\Set$ ba\-sed) logic, $\ce$ consists of surjections and $\cm$ of injections. 

It has turned out that some
well-known properties of existential quantification depend on $\ce$.
For instance, to pass $\exists$ along $\wedge$ we need that $\ce$ is stable under pullbacks~\cite{RTe2}, which is too restrictive: it
is not valid in the important case where $\cv=\Met$ is the category of metric spaces and $\ce$ consists of dense maps. The main result of \cite{RTe2}
is that the resulting enriched regular theories are precisely theories
describing enriched $\ce$-injectivity classes in the $\cv$-category of structures.

In this paper, we add (infinitary) disjunctions to our logic. This leads us
to the logic of positive-existential formulas, which we call {\em enriched positive
logic}. In the finitary case, the resulting theories are called h-inductive (see \cite{PY} for an overview). In \cite{MP}, infinitary formulas are allowed and the resulting theories are called basic.
Again, to interpret disjunctions we make use of the fixed enriched factorization system $(\ce,\cm)$ on the base $\cv$ of enrichment. Over the category
$\Ban$ of Banach spaces and linear maps of norm $\leq 1$, our finitary
enriched positive logic captures the logic of positive bounded formulas of \cite{I}.

In the same way as basic theories extend regular ones, $\ce$-cone-injectivity classes extend $\ce$-injectivity ones. Under some assumptions on $\ce$, basic theories precisely correspond to $\ce$-cone-injectivity classes. This generalizes the well known $\Set$-based correspondence where, moreover, we precisely capture
accessible categories. This extends to the enriched framework as well,
again under some assumptions about $\ce$:

\begin{theo:main}
	Under the assumptions of~\ref{char}, the following are equivalent for a full subcategory $\ca$ of $\Str(\mathbb L)$:
	\begin{enumerate}
		\item $\ca$ is isomorphic to $\Mod(\mathbb T)$ of models of a basic $\mathbb L$-theory $\mathbb T$;
		\item $\ca$ is closed under $\lambda$-filtered colimits, powers by $\ce$-stable $\cv$-connected objects, and $\lambda$-pure subobjects, for some $\lambda$;
		\item $\ca$ is accessible, accessibly embedded, and closed under powers by $\ce$-stable $\cv$-connected objects;
		\item $\ca$ is a $(\lambda,\ce)$-cone-injectivity class of $\Str(\mathbb L)$, for some $\lambda$.
	\end{enumerate}
\end{theo:main}
 
Our main examples of bases of enrichment were listed in \cite[2.3]{RTe2}
and include $\Met$ (metric spaces with distances $\infty$ allowed and
non-expanding maps as morphisms), $\Ban$ (Banach spaces  and linear maps
of norm $\leq 1$), $\Pos$ (posets and monotone maps), $\Cat$ (small
categories) and $\Ab$ (abelian groups).

In Appendix~\ref{elementary-section} we study when elementary morphisms, those that reflect the validity of positive-primitive formulas, also reflect the validity of positive-existential formulas. This is always the case over $\cv=\Set$, but in general requires additional assumption.

In the ordinary case, the fragment of regular logic that only uses unique existential quantification (that is, finite limit or cartesian theories) is used to characterize orthogonality classes, and locally presentable categories~\cite[5.29]{AR}. We treat an enriched analogue of this in Appendix~\ref{limit-section} and prove similar results. 

While ordinarily unique existence can be described by a regular theory; this does not extend to the enriched framework in full generality. The reason is that while ``being a monomorphism'' (which is needed to have uniqueness) is expressible in ordinary regular logic, ``being a map in $\cm$'' may not be expressible in enriched regular logic. Nonetheless, we give hypotheses on $(\ce,\cm)$ so that limit theories can be captured under the regular ones.

\section{Preliminaries}
Like in \cite{RTe2}, as our base of enrichment we fix  a symmetric monoidal closed category $\cv=(\cv_0,\otimes,I)$ which is locally $\lambda$-presentable as a closed category for some $\lambda$ (see \cite{Kel82}). This means that $\cv_0$ is locally $\lambda$-presentable and the full subcategory $\cv_\lambda$ spanned by the $\lambda$-presentable objects contains the unit and is closed under the monoidal structure of $\cv_0$. We denote by $$A^X:=[X,A]$$
the internal hom in $\cv$. For all other enriched notions, such as limits and colimits, we follow Kelly's book~\cite{K}.

We further assume that $\cv$ is equipped with an {\em enriched factorization system }
$$(\ce,\cm)$$ 
in the sense of \cite{LW}; that is, $(\ce,\cm)$ is an orthogonal factorization system on $\cv_0$, and $\ce$ is closed under tensoring with objects of $\cv$ (equivalently, $\cm$ is closed under powers). We will often assume the factorization system to be proper, meaning that $\ce$ is contained in the class of epimorphisms, and $\cm$ in that of monomorphisms.

\begin{exam}
	Following \cite{R1}, we say that morphism $f\colon A\to B$ in $\cv$ is a \textit{surjection} if $\cv_0(I,f)$ is surjective as a function between sets. We call {\em injection} a map that is right orthogonal with respect to all surjections. Under certain assumptions discussed in \cite{R1} the pair (surjection, injection) forms an enriched factorization system on $\cv$.
\end{exam}

Given the factorization system $(\ce,\cm)$, an object $X\in\cv$ is called {\em $\ce$-projective} if $\cv_0(X,-)\colon\cv_0\to\Set$ sends maps in $\ce$ to epimorphisms; similarly, $X$ is called {\em $\ce$-stable} if $$e^X\colon A^X\to B^X$$ is in $\ce$ whenever $e\colon A\to B$ is. When the unit $I$ is $\ce$-projective, every $\ce$-stable object is $\ce$-projective.

We say that a set of objects $\cg$ of $\cv_\lambda$ is an $\ce$-generator for $\cv_\lambda$ if for any $X\in\cv_\lambda$ there exists $Y\in\cg$ and a map $e\colon Y\to X$ in $\ce$. 

Below we collect a list of examples of bases of enrichment we are interested in, together with a choice of factorization system on them. For more details see \cite[Example~2.3]{RTe2}.

\begin{exams}
	$ $\begin{enumerate}
		\item $\cv=\Set$ with (surjection, injections). We recover the usual notions considered in ordinary positive logic.
		\item $\cv=\Met$ the category of generalized metric spaces (with distance $\infty$ allowed) and contractions. Here we have at least two interesting choices of factorization systems: (surjection, isometry) and (dense, closed isometry). Similarly, we can consider $\cv=\CMet$, the category of complete metric spaces, with the (dense, isometry) factorization system.
		\item $\cv=\Ban$ the category of Banach spaces and linear maps of norm $\leq 1$. Here we can consider the (dense, isometry) factorization system as well as the (strong epi, mono) factorization system (where monomorphisms are just one-to-one maps).
		\item $\cv=\Pos$, the category of partially ordered sets, with the (surjection, injection) factorization system (the injections here correspond to order-reflecting maps, also called embeddings). Also $\cv=\omega$-$\CPO$, the category of posets with joins of $\omega$-chains, with the (dense, closed embedding) factorization system.
		\item $\cv$ any locally presentable base of enrichments which is also a regular category, together with the (regular epi, mono) factorization system. This includes any Grothendieck topos as well as more algebraic examples such as the categories $\Ab$ of abelian groups, $R$-$\Mod$ of modules over a ring $R$, and $\mathbf{DGAb}$ of differentially graded abelian groups.
	\end{enumerate}
\end{exams}

%In \cite{RTe1}, we introduced $(\lambda,\ce)$-pure morphisms in a $\cv$-category $\ck$. Although $(\lambda,\ce)$-pure morphisms do not need to be monomorphisms in general, we said that a subcategory $\cl$ of $\ck$ is closed under $(\lambda,\ce)$-pure subobjects if for any $(\lambda,\ce)$-pure $f:K\to L$ with $L\in\cl$, then also $K\in\cl$.

\section{Basic theories}
In \cite{RTe2}, we have developed enriched regular logic where only positive-primitive formulas are needed. Now, we introduce the enriched concept of disjunctions. For that, we fix an enriched factorization system $(\ce,\cm)$ on $\cv$ whose right class is made from monomorphisms and is closed in $\cv^\to$ under $\lambda$-directed colimits.

Let us begin by recalling the notion of a language $\LL$:

\begin{defi}[\cite{RTe2}]
	A (single-sorted) {\em language} $\mathbb L$ (over $\cv$) consists of the following data:\begin{enumerate}
		\item a set of function symbols $f\colon(X,Y)$ whose arities $X$ and $Y$ are objects of $\cv$;
		\item a set of relation symbols $R:X$, with arity an object $X$ of $\cv$.
	\end{enumerate}
	The language $\mathbb L$ is called {\em $\lambda$-ary} if all the arities appearing in $\mathbb L$ lie in $\cv_\lambda$.
\end{defi}

And of an $\LL$-structure:

\begin{defi}[\cite{RTe2}]
	Given a language $\mathbb L$, an {\em $\mathbb L$-structure} consists of an object $A\in\cv$ together with:\begin{enumerate}
		\item a morphism $f_A\colon A^X\to A^Y$ in $\cv$ for any function symbol $f\colon(X,Y)$ in $\mathbb L$;
		\item an $\cm$-subobject $r_A\colon R_A\rightarrowtail A^X$ for any relation symbol $R\colon X$ in $\mathbb L$.
	\end{enumerate}
	A {\em morphism of $\mathbb L$-structures} $h\colon A\to B$ is determined by a map $h\colon A\to B$ in $\cv$ making the following square commute
	\begin{center}
		\begin{tikzpicture}[baseline=(current  bounding  box.south), scale=2]
			
			\node (a0) at (0,0.8) {$A^X$};
			\node (b0) at (1,0.8) {$B^X$};
			\node (c0) at (0,0) {$A^Y$};
			\node (d0) at (1,0) {$B^Y$};
			
			\path[font=\scriptsize]
			
			(a0) edge [->] node [above] {$h^X$} (b0)
			(a0) edge [->] node [left] {$f_A$} (c0)
			(b0) edge [->] node [right] {$f_B$} (d0)
			(c0) edge [->] node [below] {$h^Y$} (d0);
		\end{tikzpicture}	
	\end{center} 
	for any $f\colon(X,Y)$ in $\mathbb L$, and such that there exist a (necessarily unique) morphism $h_R\colon R_A\to R_B$ completing the diagram below
	\begin{center}
		\begin{tikzpicture}[baseline=(current  bounding  box.south), scale=2]
			
			\node (a0) at (0,0.8) {$R_A$};
			\node (b0) at (1,0.8) {$R_B$};
			\node (c0) at (0,0) {$A^X$};
			\node (d0) at (1,0) {$B^X$};
			
			\path[font=\scriptsize]
			
			(a0) edge [dashed, ->] node [above] {$h_R$} (b0)
			(a0) edge [>->] node [left] {$r_A$} (c0)
			(b0) edge [>->] node [right] {$r_B$} (d0)
			(c0) edge [->] node [below] {$h^X$} (d0);
		\end{tikzpicture}	
	\end{center} 
	for any relation symbol $R:X$ in $\mathbb L$.
\end{defi}

The $\cv$-category of $\LL$-structures $\Str(\LL)$ is defined in~\cite[Definition~3.8]{RTe2} and its underlying category has $\LL$-structures as objects and morphisms of such as arrows. This is locally $\lambda$-presentable as a $\cv$-category whenever $\LL$ is $\lambda$-ary and the forgetful functor
$$ \Str(\LL)\longrightarrow \cv $$
is continuous (so that it preserves products, equalisers, and powers) and preserves $\lambda$-filtered colimits.

The class of enriched terms is defined as follows:

\begin{defi}[{\cite[Definition~4.1]{RTe}}]
	Let $\mathbb L$ be a $\lambda$-ary language over $\cv$ and $\lambda\leq \kappa\leq\infty$.
	The class of \textit{$\kappa$-ary $\mathbb L$-terms} is defined recursively as follows:
	\begin{enumerate}
		\item Every morphism $f\colon Y\to X$ of $\cv_\kappa$ is an $(X,Y)$-ary term; 
		\item Every function symbol $f:(X,Y)$ of $\mathbb L$ is an $(X,Y)$-ary term;
		\item If $t$ is a $(X,Y)$-ary term and $Z$ is in $\cv_\kappa$, then $t^Z$ is a $(Z\otimes X,Z\otimes Y)$-ary term;
		\item Given $t_J=(t_j)_{j\in J}$, where $|J|<\kappa$ and $t_j$ is an $(X_j,Y_j)$-ary term, and $s$ an $(\sum_{j\in J} Y_j,W)$-ary term; then $s(t_J)$ is a $(\sum_{j\in J} X_j, W)$-ary term.
	\end{enumerate}
	If $A$ is an $\mathbb L$-structure, then the {\em interpretation} of $\mathbb L$-terms is defined recursively as follows:
	\begin{enumerate}
		\item a morphism $f\colon Y\to X$ of $\cv_\lambda$ is interpreted as $$f_A:= A^f\colon A^X\to A^Y;$$
		\item a function symbol $f:(X,Y)$ of $\mathbb L$ is interpreted as the map $$f_A\colon A^X\to A^Y$$ given by the fact that $A$ is an $\mathbb{L}$-structure;
		\item If $t$ is a $(X,Y)$-ary term and $Z$ is an arity, then $t^Z$ is interpreted as the map 
		$$ t^Z_A\colon A^{Z\otimes X}\to A^{Z\otimes Y}$$
		given by composing $(t_A)^Z\colon (A^X)^Z\to (A^Y)^Z$ with the canonical isomorphisms;
		\item If $t_J=(t_j)_{j\in J}$ is formed by $(X_j,Y_j)$-ary terms, and $s$ is a $(\sum_{j\in J}Y_j,W)$-ary term, then 
		$s(t_J)$ is interpreted as the composite
		$$ A^{\sum_{j\in J} X_j}\cong \textstyle\prod_i A^{X_j}\xrightarrow{\ \prod_j(t_j)_A\ } \textstyle\prod_j A^{Y_j}\cong A^{\sum_{j\in J}Y_j}\xrightarrow{\ s_A\ } A^W.$$
	\end{enumerate}	
\end{defi}

\begin{nota}
	As in~\cite{RTe2}, variables have arities which are objects of $\cv$ and we denote them as $x:X$, for $X\in\cv$.
\end{nota}

As mentioned before, in this section we extend regular logic to the positive-existential fragment, using disjunctions. The regular part of the following definition is taken from \cite[Definition~4.2]{RTe2}.

\begin{defi}
	Given a language $\LL$, the atomic formulas of $\LL$ are defined as follows:
	\begin{enumerate}
		\item if $s,t$ are $(X,Y)$-ary terms, then 
		$$\varphi(x):=(s(x)\eqcirc t(x))$$ is an $X$-ary atomic formula;
		\item if $R$ is a $X$-ary relation symbol, $Y$ is an arity, and $t$ an $(Z,X\otimes Y)$-ary term, then 
		$$\varphi(z):= R^Y(t(z))$$ is a $Z$-ary atomic formula.
	\end{enumerate}
	General formulas are built recursively from atomic formulas by taking:
	\begin{enumerate}
		\item[(3)] {\em conjunctions}: if $\varphi_j(x)$, for $j\in J$, are $X$-ary formulas, then  
		$$\varphi(x):=\bigwedge\limits_{j\in J}\varphi_j$$ is an $X$-ary formula;
		\item[(4)] {\em existential quantification}: if $\psi(x,y)$ is an $X+Y$-ary formula, then
		$$ \varphi(x):= (\exists y)\psi(x,y) $$
		is an $X$-ary formula;
		\item[(5)] {\em disjunctions}: if $\varphi_j$, for $j\in J$, are $X$-ary formulas, then
		$$
		\varphi(x):=\bigvee\limits_{j\in J}\varphi_j
		$$
		is an $X$-ary formula.
	\end{enumerate}
	If $\LL$ is $\lambda$-ary, the arities involved are $\lambda$-presentable, and $|J|<\kappa$, we get $\mathbb L_{\kappa\lambda}$-formulas. 
\end{defi}

Next we define interpretation of formulas in an $\LL$-structure. Again, the regular part of the definition is taken from~\cite[Definition~4.5]{RTe2}.

\begin{defi}
	For any $\LL$-structure $A$ and any $X$-ary formula $\varphi(x)$ we define its interpretation in $A$ as an $\cm$-subobject
	$$ \varphi_A\rightarrowtail A^X.$$
	We argue recursively as follows:
	{\setlength{\leftmargini}{1.6em}
		\begin{enumerate}
			\item If $\varphi(x):=(s(x)\eqcirc t(x))$, then the interpretation $\varphi_A\rightarrowtail A^X$ is defined as the $(\ce,\cm)$-factorization of the 	the equalizer of $s_A$ and $t_A$:
			\begin{center}
				\begin{tikzpicture}[baseline=(current  bounding  box.south), scale=2]
					
					\node (a0) at (0,0) {$A_{s,t}$};
					\node (c0) at (1.8,0) {$ A^X $};
					\node (d0) at (0.9,0.5) {$\varphi_A$};
					
					\path[font=\scriptsize]
					
					(a0) edge [->] node [below] {$\tx{eq}(s_A,t_A)$} (c0)
					(a0) edge [->>] node [above] {$\ce\ \ \ $} (d0)
					(d0) edge [>->] node [above] {$\ \ \ \cm$} (c0);
				\end{tikzpicture}	
			\end{center}
			(If $\cm$ contains the regular monomorphisms, then $\varphi_A=A_{s,t}$.)
			\item If $\varphi(z):= R^Y(t(z))$, then the interpretation $\varphi_A\rightarrowtail A^Z$ is defined by the pullback
			\begin{center}
				\begin{tikzpicture}[baseline=(current  bounding  box.south), scale=2]
					
					\node (a0) at (0,0.8) {$\varphi_A$};
					\node (a0') at (0.2,0.6) {$\lrcorner$};
					\node (b0) at (1,0.8) {$A^Z$};
					\node (c0) at (0,0) {$R_A^Y$};
					\node (d0) at (1,0) {$A^{X\otimes Y}$};
					
					\path[font=\scriptsize]
					
					(a0) edge [>->] node [above] {} (b0)
					(a0) edge [->] node [left] {} (c0)
					(b0) edge [->] node [right] {$t_A$} (d0)
					(c0) edge [>->] node [below] {$r_A^Y$} (d0);
				\end{tikzpicture}	
			\end{center}
			where we composed $r_A^Y$ with the isomorphism $(A^X)^Y\cong A^{X\otimes Y}$  %(note, the top arrow is still in $\cm$ since it is closed under powers and pullbacks in $\cv^\to$).
			\item If $ \varphi(x):=\bigwedge_{j\in J}\varphi_j$, then $\varphi_A\rightarrowtail A^X$ is the intersection (wide pullback) of all the $(\varphi_j)_A\rightarrowtail A^X$ for $j\in J$. %(Since $\cm$ is closed under intersections, this map is in $\cm$).
			\item If $\varphi(x):=((\exists y)\psi(x,y))$, then $\varphi_A\rightarrowtail A^X$ is given by the $(\ce,\cm)$-factorization below
			\begin{center}
				\begin{tikzpicture}[baseline=(current  bounding  box.south), scale=2]
					
					\node (a0) at (-0.1,0.8) {$\psi_A$};
					\node (c0) at (1,0.8) {$ A^{X+Y} $};
					\node (c1) at (2.1,0.8) {$ A^X$};
					\node (d0) at (1,1.3) {$\varphi_A$};
					
					\path[font=\scriptsize]
					
					(a0) edge [>->] node [below] {} (c0)
					(c0) edge [->] node [below] {$p_1$} (c1)
					(a0) edge [->>] node [above] {$\ce\ \ \ $} (d0)
					(d0) edge [>->] node [above] {$\ \ \ \cm$} (c1);
				\end{tikzpicture}	
			\end{center}
			where the bottom composite is given by the $\cm$-morphism defining $\psi_A$ and the projection on the first factor.
			\item if $\varphi(x):= \bigvee\limits_{j\in J}\varphi_j(x)$, then $\varphi_A\rightarrowtail A^X$ is given by the $(\ce,\cm)$-factorization 
			\begin{center}
				\begin{tikzpicture}[baseline=(current  bounding  box.south), scale=2]
					
					\node (a0) at (0,0) {$\textstyle\sum_{j\in J}(\varphi_j)_A$};
					\node (c0) at (1.8,0) {$ A^X $};
					\node (d0) at (0.9,0.5) {$\varphi_A$};
					
					\path[font=\scriptsize]
					
					(a0) edge [->] node [below] {} (c0)
					(a0) edge [->>] node [above] {$\ce\ \ \ $} (d0)
					(d0) edge [>->] node [above] {$\ \ \ \cm$} (c0);
				\end{tikzpicture}	
			\end{center}
			of the morphism induced by the interpretations $(\varphi_j	)_A\rightarrowtail A^X$ of $\varphi_j$, $j\in J$, in $A$.
	\end{enumerate}}
\end{defi}	

\begin{nota}
	For a fixed arity $X$, the empty conjunction is denoted by $\top$, and means \textit{true}, while the empty disjunction is denoted as $\bot$, and means \textit{false}. For an $\LL$-structure $A$ we always have $\top_A=1_{A^X}=A^X\to A^X$ and $\bot_A=\hat 0\rightarrowtail A^X$, given by the $(\ce,\cm)$-factorization of $0\to A^X$. In all our examples we actually have $\bot_A=0$; we do not actually require all initial maps $0\to X$ in $\cv$ to lie in $\cm$ since that would not simplify any calculation.
\end{nota}

Recall that a formula is \textit{positive-primitive} if it has the form $(\exists x)\varphi$ where $\varphi$ is a conjunction of atomic formulas (see \cite[4.3]{RTe2}). 

\begin{defi}
	A \textit{positive-existential} formula is one of the form
	$$ \varphi(x):= \bigvee\limits_{j\in J}\varphi_j(x) $$
	where each $\varphi_j$ is positive-primitive.
\end{defi}

Recall that an object $X\in\cv$ is called $\cv$-connected if $(-)^X=[X,-]$ preserves coproducts.

\begin{lemma}\label{stable} 
The following properties hold:
{\setlength{\leftmargini}{1.6em}
	\begin{enumerate}
		 
		\item If $X\in\cv$ is an $\ce$-stable $\cv$-connected object, then for every positive-existential formula $\varphi$ 
		$$
		\varphi_{A^X}\cong(\varphi_{A})^X.
		$$
		\item If $A\cong\colim A_i$ is a $\lambda$-directed colimit of $\LL$-structures, then
		$$
		\varphi_A\cong\colim\varphi_{A_i}
		$$
		for every $\lambda$-ary positive-existential $\varphi$.
	\end{enumerate}}
\end{lemma}
\begin{proof}
Following \cite[4.8]{RTe2}, it holds for positive-primitive formulas.
It extends to positive-existential ones because 
$(\ce,\cm)$-factorizations are stable under the limits and colimits considered in the statement (given the additional hypotheses).
\end{proof}

As for the regular fragment~\cite[Section~4]{RTe2}, we can talk about satisfaction and pointwise satisfaction of positive existential formulas, and sequents of such:

\begin{defi}
	Given an $\mathbb L$-structure $A$, an $X$-ary positive-existential formula $\varphi(x)$, and an arrow $a\colon I\to A^X$ in $\cv$, we say that {\em $A$ satisfies $\varphi[a]$} and write
	$$ A\models\varphi[a]  $$
	if $a$ factors through $\varphi_A\rightarrowtail A^X$.
\end{defi}

Then we have two possible approaches to satisfaction:
\begin{defi}
		Given an $\mathbb L$-structure $A$ and $X$-ary positive-existential formula $\varphi(x)$, we say that $A\in\Str(\LL)$ {\em satisfies $\varphi(x)$}, and write
	$$A\models \varphi,$$ 
	if $\varphi_A\rightarrowtail A^X$ is an isomorphism. We say that {\em $A$ satisfies $\varphi(x)$ pointwise} if $ A\models\varphi[a] $ for any $a\colon I\to A^X$.
\end{defi}
	 
Clearly, satisfaction implies pointwise satisfaction. The converse holds if $\ce$ contains the class of surjections: if $A$ satisfies $\varphi(x)$ pointwise, then the $\cm$-subobject $\varphi_A\rightarrowtail A^X$ is (a surjection and thus) in $\ce$, so that $\varphi_A\cong A^X$. 

\begin{defi}
	Given $X$-ary positive-existential formulas $\varphi$ and $\psi$, we say that $A\in\Str(\LL)$ {\em satisfies the sequent} $
	 \varphi(x) \vdash \psi(x),
	$ and write
	$$ A\models( \varphi\vdash \psi), $$
	if $\varphi_A\rightarrowtail A^X$ factors (necessarily uniquely) through  $\psi_A\rightarrowtail A^X$. We say that $A$ {\em satisfies the sequent pointwise} if for any $a$ such that $A\models\varphi[a]$, then 
	$A\models\psi[a]$.
\end{defi}

Again, satisfaction implies pointwise satisfaction. If $\ce$ contains the surjections, the two notions are equivalent (we argue as in~\cite{RTe2} after Remark~4.13): consider the pullback below;
\begin{center}
	\begin{tikzpicture}[baseline=(current  bounding  box.south), scale=2]
		
		\node (a0) at (0,0.8) {$(\varphi\wedge\psi)_A$};
		\node (a0') at (0.2,0.6) {$\lrcorner$};
		\node (b0) at (1.2,0.8) {$\varphi_A$};
		\node (c0) at (0,0) {$\psi_A$};
		\node (d0) at (1.2,0) {$A^X$};
		
		\path[font=\scriptsize]
		
		(a0) edge [>->] node [above] {$f$} (b0)
		(a0) edge [>->] node [left] {} (c0)
		(b0) edge [>->] node [right] {} (d0)
		(c0) edge [>->] node [below] {} (d0);
	\end{tikzpicture}	
\end{center}
if $A$ satisfies the sequent pointwise, then $f$ is easily seen to be a surjection, and hence lies in $\ce$. But $f$ is also in $\cm$ by stability under pullbacks; thus $f$ is an isomorphism and so $\varphi_A\rightarrowtail A^X$ factors through  $\psi_A\rightarrowtail A^X$.

\begin{defi}
	Given a set $\TT$ of sequents of positive-existential formulas, we denote by $\Mod(\TT)$ the full subcategory of $\Str(\LL)$ spanned by those $\LL$-structures that satisfy the sequents in $\TT$. Such sets of sequents are called \textit{basic} theories.
\end{defi}

\begin{rem}
	Note that by taking $\top$ on the left-hand-side of a basic sequent, we are allowed to consider axioms of the form 
	$$ (\forall x)\psi(x) $$
	where $\psi(x)$ is positive existential. An $\LL$-structure satisfies this if and only if the composite 
	$$ \psi_A\rightarrowtail A^X $$
	is an isomorphism. 
\end{rem}

Every regular theory from \cite{RTe2} is basic. Note moreover that for any $A\in\Str(\LL)$ and a set of positive-primitive formulas $(\psi_j(x))_{j\in J}$ we have 
$$ A\models \psi_{j_0}\vdash \bigvee_{j\in J}\psi_j $$
 for any $j_0\in J$ (this follows directly from how disjunction is interpreted).

\begin{propo}
	The following are equivalent for an $\LL$-structure $A$, a positive-existential formula $\varphi(x)$, and a set of positive-primitive formulas $(\psi_j(x))_{j\in J}$:\begin{enumerate}
		\item $A\models (\bigvee_{j\in J} \psi_j)\vdash \varphi$;
		\item $A\models \psi_j\vdash\varphi$, for any $j\in J$.
	\end{enumerate}
\end{propo}
\begin{proof}
	Consider the solid part of the following diagram.
	\begin{center}
		\begin{tikzpicture}[baseline=(current  bounding  box.south), scale=2]
			
			\node (a0) at (0,1.6) {$\sum_{j\in J}(\psi_j)_A$};
			\node (a0') at (1.1,0.8) {$(\bigvee_{j\in J}\psi_j)_A$};
			\node (b0) at (0,0.8) {$\varphi_A$};
			\node (c0) at (-1,0.8) {$(\psi_j)_A$};
			\node (d0) at (0,0) {$A^X$};
			
			\path[font=\scriptsize]
			
			(a0) edge [dashed, ->] node [right] {$m$} (b0)
			(c0) edge [->] node [left] {$\iota_j$} (a0)
			(b0) edge [>->] node [right] {} (d0)
			(c0) edge [>->] node [below] {} (d0)
			(a0') edge [>->] node [below] {} (d0)
			(c0) edge [dashed, >->] node [below] {$m_j$} (b0)
			(a0) edge [->>] node [right] {$\ \ e\in\ce$} (a0');
		\end{tikzpicture}	
	\end{center}
	If (1) holds, then the map $m$ above exists (being the composite of $e$ with the inclusion $(\bigvee_{j\in J}\psi_j)_A\subseteq \varphi_A$), and thus also the maps $m_j$ exist, showing that (2) holds. Conversely, if (2) holds, so that the maps $m_j$ above exist, then also $m$ exist by the universal property of coproducts. Then, by orthogonality of the factorization system, (1) holds.
\end{proof}

This implies that basic theories can be equivalent presented by sequents
$$\varphi(x) \vdash \psi(x)$$ where $\varphi$ is positive-primitive and $\psi$ positive-existential.

\begin{rem}
	Assume that the terminal object $1$ is connected and $\ce$-projective, and let $\{\varphi_j\}_{j\in J}$ be a set of statements (that is, formulas of arity $0$). Then, for any $\LL$-structure $A$, if $A\models \bigvee_{j\in J}\varphi_j$, then there exists $j_0\in J$ for which $A\models \varphi_{j_0}$. Indeed, if $A$ proves the disjunction of all the $\varphi_j$, then we can consider the following composite
	$$ \sum_{j\in J}(\varphi_j)_A\xrightarrow{\ \ce\ }(\bigvee_{j\in J}\varphi_j)_A\xrightarrow{\ \cong\ } A^0\cong 1 $$
	in $\cv$, where the last map is an isomorphism. By $\ce$-projectivity and connectedness of $1$, the identity map into the terminal object factors through one of the inclusions $\varphi_{j_0}\rightarrowtail 1$, which is then a split epimorphism and hence an isomorphism. Thus  $A\models \varphi_{j_0}$.
\end{rem}

\begin{exam}\label{metric}
	Consider $\cv=\Met$ with the (surjective, isometry) factorization system (but also the (dense, closed isometry) works); see \cite[3.16]{AR1} for these factorization systems. Then $X\in\Met$ is a ``standard'' metric space (that is, distance has finite real values) if and only if it satisfies the $(1+1)$-ary formula $\varphi$ given below
	$$ \varphi(x,y):= \bigvee_{n\in\mathbb N} \exists (z:2_n)\ (\pi_1(z)\eqcirc x)\wedge (\pi_2(z)\eqcirc y) . $$
	Here, $2_n$ denote the metric space having two points whose distance is $n$.	
\end{exam}

\begin{exam}
	Let $\cv=\Ban$ with the (strong epimorphism, monomorphism) factorization system (but also the (dense, embedding) works). 
	Notice that strong epimorphisms coincide with surjections while dense maps coincide with epimorphisms. We will now define a theory for Hilbert spaces; this will be a two-sorted theory (it is an easy generalization to modify the definition above to include more than one sort).
	
	To begin with, it is easy to construct a two-sorted language $\LL$, with sorts $S_1,S_2$, and a regular theory $\TT$ whose models are pairs of Banach spaces $(\mathbb B,\mathbb B')$ where the underlying vector spaces satisfy $\mathbb B'=\mathbb B \oplus\mathbb B$: the language $\LL$ needs to contain two product projections $p_i,p_2$, and $\TT$ specifies that the underlying set of $\mathbb B'$ is the product of two copies of $\mathbb B$). 
	
	Now we enlarge $\mathbb T$ to impose that
	\begin{equation}\label{hilbert}
		\| (x,y)\|_{\mathbb B'}= \sqrt{2\| x\|_\mathbb B +2\| y\|_\mathbb B} 
	\end{equation}
	for any $(x,y)\in\mathbb B'$. For any $p,q\in\mathbb Q$ let $\overline{(p,q)}=\sqrt{2p^2+2q^2}$ and consider the sentence
	$$ \varphi_{p,q}:= (\forall x:\mathbb C_p) (\forall y:\mathbb C_q)(\exists z:\mathbb C_{\overline{(p,q)}})\ \  p_1z\eqcirc x \wedge p_2z\eqcirc y $$
	where $x,y$ are of sort $S_1$ and $z$ of sort $S_2$. Here,
	$\mathbb C_p$ are complex numbers with the norm $\|1\|_p=p$. It follows that $(\mathbb B,\mathbb B')$ satisfies $\varphi_{p,q}$ if and only if the inequality $\leq$ holds in \ref{hilbert}.
	
	Next, for any $p,q$ as above and $0<\epsilon<\overline{(p,q)}$, add to $\mathbb T$ the basic sentence:
	$$\psi_{p,q,\epsilon}:=  (\forall x:\mathbb C_p) (\forall y:\mathbb C_q)\ \left( (\exists z:\mathbb C_{\overline{(p,q)}-\epsilon}) (p_1z\eqcirc x \wedge p_2z\eqcirc y) \vdash \bot \right). $$
	This says that the norm of $(x,y)$ cannot be smaller than the right-hand-side of \ref{hilbert}. 
	
	Since $\mathbb B'$ with the norm of \ref{hilbert} is exactly the coproduct of $\mathbb B$ with itself in $\Ban$, it follows that a model of $\TT$ is the same as a pair $(\mathbb B,\mathbb B+\mathbb B)$, where $\mathbb B\in\Ban$ (in particular $\Mod(\TT)\simeq \Ban$ at this stage).
	
	Finally, we enlarge $\TT$ with an equations specifying that the map (definable in our language)
	$$ f(x,y)= ((x+y)/\sqrt 2, (x-y)/\sqrt 2)$$
	is an involution: $f\circ f=1$. It is easy so see that a Banach space satisfies this property if and only if it satisfies the parallelogram law, if and only if it is an Hilbert space. 
\end{exam}

\begin{rem}
	(1) In \cite[4.16]{RTe2}, we have compared enriched regular logic over $\Ban$ with the logic of positive bounded formulas of \cite{I}. The conclusion was that the satisfaction from \cite{I} is our satisfaction with respect to the factorization system (strong epimorphism, monomorphism) while the approximate satisfaction from \cite{I} coincides with our satisfaction with respect to the factorization system (epimorphisms, strong monomorphisms). The presence of disjunctions does
	not change this picture because the interpretation of {\em finite} disjunctions
	is the same for both factorization systems. 
	
	(2) Similarly, the interpretation of finite disjunctions in $\Met$
	is the same for both factorization systems from \ref{metric}.
\end{rem}

\section{Elementary morphisms}

In this section we introduce the notion of p.e. elementary morphism of structure, as a stronger condition than that of elementary morphism. We shall see later that in many cases they coincide.
Before stating the definition we note the following useful result.

\begin{propo}\label{psi-functorial}
	Let $\psi$ be a positive formula in $\mathbb L_{\lambda\lambda}$ and $f\colon K\to L$ a morphism in $\Str(\mathbb L)$, then there exists a necessarily unique morphism as dashed below
	\begin{center}
		\begin{tikzpicture}[baseline=(current  bounding  box.south), scale=2]
			
			\node (a0) at (0,0.8) {$\psi_K$};
			\node (b0) at (1,0.8) {$\psi_L$};
			\node (c0) at (0,0) {$K^X$};
			\node (d0) at (1,0) {$L^X$};
			
			\path[font=\scriptsize]
			
			(a0) edge [dashed, ->] node [above] {$\psi_f$} (b0)
			(a0) edge [>->] node [left] {} (c0)
			(b0) edge [>->] node [right] {} (d0)
			(c0) edge [->] node [below] {$f^X$} (d0);
		\end{tikzpicture}	
	\end{center}
	making the square commute. This assignment defines a functor
	$$ \psi_{(-)}\colon \Str(\LL)_0\longrightarrow\cv_0 $$
	which preserves $\lambda$-filtered colimits.
\end{propo}
\begin{proof}
	This is done by induction on the complexity of $\psi$. The case where $\psi$ is a conjunction of atomic formula has already been observed in \cite[Section~5]{RTe2} where the functor is actually induced by a $\cv$-functors which is continuous and $\lambda$-filtered colimits preserving.
	
	Existence and functoriality of the dashed arrow in the case where $\psi$ is obtained by taking existential quantification or disjunctions of simpler formulas, follows from functoriality of the factorization system and of coproducts. Preservation of $\lambda$-filtered colimits follows from the fact that $\ce$ and $\cm$ are stable under them, and so are coproducts.
\end{proof}

We recall the notion of $\lambda$-elementary morphism from~\cite[Definition~6.3]{RTe2} and extend it to the case of positive-existential formulas. 

\begin{defi}
	A morphism $f\colon K\to L$ in $\Str(\mathbb L)$ is {\em elementary} with respect to a positive-existential formula $\psi$ if the diagram 
\begin{center}
	\begin{tikzpicture}[baseline=(current  bounding  box.south), scale=2]
		
		\node (a0) at (0,0.8) {$\psi_K$};
		\node (a0') at (0.2,0.6) {$\lrcorner$};
		\node (b0) at (1,0.8) {$\psi_L$};
		\node (c0) at (0,0) {$K^X$};
		\node (d0) at (1,0) {$L^X$};
		
		\path[font=\scriptsize]
		
		(a0) edge [->] node [above] {} (b0)
		(a0) edge [>->] node [left] {} (c0)
		(b0) edge [>->] node [right] {} (d0)
		(c0) edge [->] node [below] {$f^X$} (d0);
	\end{tikzpicture}	
\end{center}
is a pullback. 
We say that $f$ is {\em $\lambda$-elementary} if it is elementary with respect to every positive-primitive formula $\psi(x)$ in 
$\mathbb L_{\lambda\lambda}$. If $f$ is elementary with respect to every positive-existential formula in 
$\mathbb L_{\lambda\lambda}$ then $f$ is called {\em $\lambda$-p.e.\! elementary}.
\end{defi}

Clearly every $\lambda$-p.e.\! elementary morphism is $\lambda$-elementary. The converse is true over $\Set$ but, as we shall see, for a general $\cv$ requires additional assumptions.

We next give examples of $\lambda$-p.e.\! elementary morphisms. Recall from~\cite{AR} that a morphism $f\colon K\to L$ in a locally $\lambda$-presentable ordinary category $\ck$ is called {\em pure} if for any commutative square
\begin{center}
	\begin{tikzpicture}[baseline=(current  bounding  box.south), scale=2]
		
		\node (a'0) at (0,0.8) {$A$};
		\node (b0) at (0.9,0.8) {$B$};
		\node (c0) at (0,0) {$K$};
		\node (d0) at (0.9,0) {$L$};
		
		\path[font=\scriptsize]

		(a'0) edge [->] node [above] {$g$} (b0)
		(b0) edge [->] node [right] {$v$} (d0)
		(c0) edge [->] node [below] {$f$} (d0)
		(a'0) edge [->] node [left] {$u$} (c0);
	\end{tikzpicture}	
\end{center}
with $A$ and $B$ $\lambda$-presentable, there exists a morphism $t:B\to K$ such that $t\circ g=u$. Enriched notions of purity have been considered in \cite{RTe1} and compared to elementary morphisms in \cite{RTe2}; however, for this section the ordinary notion will suffice.

\begin{propo}\label{examples-lambda-p.e.elementary}
	Let $\LL$ be a $\lambda$-ary language; then in $\Str(\LL)$:\begin{enumerate}
		\item every split monomorphism is $\lambda$-p.e.\! elementary;
		\item $\lambda$-p.e.\! elementary morphisms are stable under $\lambda$-filtered colimits in $\Str(\LL)^\to$;
		\item every $\lambda$-pure morphism is $\lambda$-p.e.\! elementary.
	\end{enumerate}
\end{propo}
\begin{proof}
	(1). Let $s\colon K\to L$ be a split monomorphism with left inverse $t\colon L\to K$. Then we can consider the following diagram induced by Proposition~\ref{psi-functorial} where the top and bottom forks are split equalisers.
	\begin{center}
		\begin{tikzpicture}[baseline=(current  bounding  box.south), scale=2]
			
			\node (a0) at (-2,0) {$\psi_K$};
			\node (a) at (-1,0) {$\psi_L$};
			\node (a1) at (0.2,0) {$\psi_L$};
			\node (b0) at (-2,-0.8) {$K^X$};
			\node (b) at (-1,-0.8) {$L^X$};
			\node (b1) at (0.2,-0.8) {$L^X$};

			\path[font=\scriptsize]
			
			(a) edge [>->] node [left] {$m_L$} (b)
			(a1) edge [>->] node [right] {$m_L$} (b1)
			(a0) edge [>->] node [left] {$m_K$} (b0)
			
			(a0) edge [->] node [above] {$\psi_s$} (a)
			(b0) edge [->] node [below] {$s^X$} (b)
			
			([yshift=-1.5pt]a.east) edge [->] node [below] {$1$} ([yshift=-1.5pt]a1.west)
			([yshift=1.5pt]a.east) edge [->] node [above] {$\psi_s\circ\psi_t$} ([yshift=1.5pt]a1.west)
			([yshift=-1.5pt]b.east) edge [->] node [below] {$1$} ([yshift=-1.5pt]b1.west)
			([yshift=1.5pt]b.east) edge [->] node [above] {$s^X\circ t^X$} ([yshift=1.5pt]b1.west);
		\end{tikzpicture}
	\end{center}
	It is well-known that in such situation the square on the left is a pullback if and only if $\psi_s$ is the equaliser of $(s^X\circ t^X\circ m_L,m_L)$. But this is true by the commutativity of the right-hand-side of the diagram and the fact that $m_L$ is mono. Thus $s$ is $\lambda$-p.e.\! elementary.
	
	(2). Let $f\cong \colim_{i\in\ci} f_i$ be a $\lambda$-filtered colimit of $\lambda$-p.e.\! elementary morphisms. By Proposition~\ref{psi-functorial} the square expressing the fact that $f$ is $\lambda$-p.e.\! elementary with respect to a positive-existential formula $\psi$ is the $\lambda$-filtered colimit of the corresponding squares for the $f_i$'s. Since pullbacks commute with $\lambda$-filtered colimits in a locally $\lambda$-presentable category, the result follows.
	
	(3). By \cite[Proposition~2.30]{AR} every $\lambda$-pure morphism is a $\lambda$-filtered colimit of split monomorphisms; thus the result follows from (1) and (2).
\end{proof}

The following implies that $\lambda$-p.e.\! elementary also reflect validity of sequents.

\begin{propo}\label{pos-ex-sequent}
	If $f\colon K\to L$ is $\lambda$-p.e.\! elementary, $\varphi$ and $\psi$ are positive-existential formulas, and $L$ satisfies the sequent  
	$$\varphi(x) \vdash \psi(x)$$
	in $\mathbb L_{\lambda\lambda}$ then $K$ also satisfies the same sequent.
\end{propo}
\begin{proof} 
	It follows essentially by definition. Consider the solid part of the diagram below
	\begin{center}
		\begin{tikzpicture}[baseline=(current  bounding  box.south), scale=2, on top/.style={preaction={draw=white,-,line width=#1}}, on top/.default=6pt]
			
			\node (a0) at (-0.4,0.75) {$\varphi_K$};
			\node (b0) at (0.6,0.75) {$\varphi_L$};
			
			\node (a'0) at (0.4,1.2) {$\psi_K$};
			\node (b'0) at (1.4,1.2) {$\psi_L$};
			
			\node (c0) at (0,0) {$K^X$};
			\node (d0) at (1,0) {$L^X$};
			
			\path[font=\scriptsize]

			(a'0) edge [->] node [above] {} (b'0)
			(a'0) edge [>->] node [left] {} (c0)
			(b'0) edge [>->] node [right] {} (d0)
			
			(a0) edge [->, on top] node [above] {} (b0)
			(a0) edge [>->] node [left] {} (c0)
			(b0) edge [>->] node [right] {} (d0)
			
			(b0) edge [>->] node [right] {} (b'0)
			(a0) edge [dashed, >->] node [right] {} (a'0)
			
			(c0) edge [->] node [below] {$f^X$} (d0);
		\end{tikzpicture}	
	\end{center}
	where the two vertical squares are pullbacks since $f$ is elementary with respect to $\varphi$ and $\psi$ by hypothesis, and the arrow $\varphi_L\to\psi_L$ is induced by the fact that $L$ satisfies the sequent. By the universal property of the pullback, the dashed arrow above exists, showing that also $K$ satisfies the sequent.
\end{proof}  

Below we say that $\cl\subseteq\ck$ is \textit{closed under $\lambda$-p.e.\! elementary subobjects} if for any $\lambda$-p.e.\! elementary $f\colon K\to L$ with $L\in\cl$, then also $K\in\cl$.

\begin{propo}\label{stable1} 
	Let $\mathbb T$ be a basic theory in 
	$\mathbb L_{\lambda\lambda}$. Then $\Mod(\mathbb T)$ is closed under
	\begin{enumerate}
		\item $\lambda$-directed colimits;
		\item $\lambda$-p.e.\! elementary subobjects;
		\item powers by $\ce$-stable $\cv$-connected objects.
	\end{enumerate}
\end{propo}
\begin{proof}
	(1) follows from Lemma~\ref{stable}(2).
	
	(2) follows from Proposition~\ref{pos-ex-sequent}.
	
	(3) follows from Lemma~\ref{stable}(1).
\end{proof}

%One cannot expect that categories of models of basic theories in $\mathbb L_{\lambda\lambda}$ are characterized by \ref{stable1}(1-3). 
%The question is whether categories of models of basic theories in $\mathbb L$ are characterized by closure properties from \ref{stable1}.

\section{Enriched cone-injectivity classes}

In this section we compare $\cv$-categories of models of basic theories with enriched cone-injectivity classes, that we introduce below.

\begin{defi}
Given a $\cv$-category $\ck$ and a cone $h=(h_j\colon A\to B_j)_{\in J}$ in $\ck_0$, we say that an object $K$ is $h$-injective if and only if the map
$$
\sum\limits_{j\in J} \ck(B_j,K)\to\ck(A,K)
$$
induced by $\ck(h_j,K):\ck(B_j,K )\to\ck(A,K)$, is in $\ce$.
Given a set $\ch$ of cones between $\lambda$-presentable objects in $\ck$, an object $K$ is $\ch$-injective if it is 
injective to every $h\in\ch$. Such classes of objects are called $\lambda$-cone-injectivity classes, or 
$(\lambda,\ce)$-cone-injectivity classes if we want to stress $\ce$.
\end{defi}

If we fix $\ck=\Str(\LL)$ and a cone $h=(h_j\colon A\to B_j)_{\in J}$ between $\lambda$-presentable $\mathbb L$-structures, we say that a regular sequent $\iota_h$ of $\mathbb L_{\lambda\lambda}$ is a \textit{cone-injectivity sequent} for $h$ if a $\mathbb L$-structure $K$ is $h$-cone-injective if and only if $K\models\iota_h$. 

In order to fully relate basic theories and enriched cone-injectivity, we will need presentation formulas of $\mathbb L$-structures. In~\cite[Section~5]{RTe2} we show that a conjunction of atomic formulas $\varphi$ is presented by an $\LL$-structure $A$ (or that $\varphi$ is a presentation formula for $A$) if we have a natural isomorphism
$$ \Str(\LL)(A,-)\cong \varphi_{(-)}\colon \Str(\LL)\to \cv.$$
The existence of presentation formulas for a given $\LL$-structure is what makes the results about injectivity hold in \cite[Section~7]{RTe2}. Such existence is not guaranteed without further assumptions (\cite[6.1]{RTe2}). Since we use the existence of presentation formulas also in this paper, for the reader's convenience, we reproduce these assumptions here. Note that these will only be used when explicitly stated.

\begin{assume}\label{assumption} 
	We fix a proper enriched factorization system $(\ce_0,\cm_0)$ on $\cv$ for which $\cm_0$ is closed in $\cv^\to$ under $\lambda$-filtered colimits, and we assume that either of the following two conditions holds:\begin{itemize}
		\item $\LL=\RR$ is a $\lambda$-ary relational language;
		\item $\mathbb L=$ is a $\lambda$-ary language whose function symbols have $\ce_0$-stable input arities and relation symbols have $\ce_0$-stable arities.  Moreover $\cv_\lambda$ has an $\ce_0$-generator of $\ce_0$-projective objects.
	\end{itemize}
We also fix an enriched factorization system $(\ce,\cm)$, such that $\ce_0\subseteq\ce$ and for which the inclusion $ J \colon \Str(\mathbb L)\hookrightarrow\Str(\mathbb L)^0$ preserves $\lambda$-presentable objects. Structures and formulas will be considered with respect to the factorization system $(\ce,\cm)$.\\
\end{assume}

For a guide on how to use this assumption see the beginning of \cite[Section~6]{RTe2} and the remark below.

\begin{rem}
	The assumption holds for any $\cv$ whenever the language is purely relational (one takes $(\ce_0,\cm_0)=(\ce,\cm)$). The second condition is valid whenever the two factorization systems on $\cv$ coincide with $(Surj,Inj)$ by considering only discrete arities; or when $\cv$ is a symmetric monoidal quasivariety with the (regular epi, mono) factorization system and $\LL$ has regular-projective arities. So that in most cases the two factorization systems can be taken to be the same; this is not the case for $\Met$ and $\Ban$ where we take $(\ce_0,\cm_0)$=(Surj,Inj) and $(\ce,\cm)$=(dense, isometry) over languages whose arities are discrete.
	
	Including $\Met$ and $\Ban$ amongst the class of examples, was exactly the reason why we weakened the assumption to allow two different factorization systems.
\end{rem}

Recall from~\cite[Remark~4.14]{RTe2} that, if we are not assuming pullback stability of $\ce$, the Frobenius condition need not hold, meaning that not every structure satisfies the sequent
$$  \varphi(x)\wedge(\exists y)\psi(x,y)\ \vdash\ (\exists y)(\varphi(x)\wedge\psi(x,y)).$$
It follows that asking for the satisfaction of such a sequent is a non-trivial condition. This is exactly the situation of the result below, where we characterize $\ce$-cone-injectivity classes in terms of certain basic theories where the formula on the left of the sequent is also repeated on the right-hand-side. (This is a similar situation to that of \cite[Theorem~7.6]{RTe2}.)

\begin{propo}\label{mod<->inj}
Consider the following conditions for a given full subcategory $\ca$ of $\Str(\LL)$:\begin{enumerate}
	\item $\ca$ is a $(\lambda,\ce)$-cone-injectivity class in $\Str(\LL)$;
	\item $\ca\cong\Mod(\TT)$ for a $\mathbb L_{\lambda\lambda}$-theory $\TT$ whose axioms are of the form
	$$ \varphi(x) \vdash \bigvee_{j\in J}(\exists y_j)(\psi_j(x,y_j)\wedge \varphi(x)),$$
	where $\varphi$ and $\psi$ are conjunctions of atomic formulas of
	$\mathbb L_{\lambda\lambda}$.
\end{enumerate}
Then $(2)\Rightarrow(1)$ always holds, and $(1)\Rightarrow (2)$ holds under Assumption~\ref{assumption}.
\end{propo}
\begin{proof}
$(1)\Rightarrow(2)$. Assume \ref{assumption} and let $(h_j\colon A\to B_j)_{j\in J}$ be a cone of $\lambda$-presentable $\mathbb L$-structures; then we can consider $\pi^A(x)$, $\pi_j^B(y)$, and $\chi_j(x)$ as in \cite[5.11]{RTe2}. We shall prove that the sequent 
$$ 
\iota_h:= \pi^A(x)\vdash \bigvee\limits_{j\in J}\chi_j(x) 
$$
is a cone-injectivity sequent for $h$.  Note that the other sequent is satisfied by any $\mathbb L$-structure $K$ using \cite[5.11]{RTe2}.
	
Let $K$ be any $\mathbb L$-structure; then $K$ satisfies $\iota_h$ if and only if  $$\bigvee\limits_{j\in J}(\chi_j)_K\cong \pi^A_K$$ as $\cm$-subobjects of $K^{Z_A}$, where $Z_A$ is from the proof of \cite[5.11]{RTe2} (since $K$ always satisfies the other sequent), if and only if the map 
$$
\sum\limits_{j\in J}\Str(\mathbb L)(B_j,K)\to\Str(\mathbb L)(A,K)
$$ 
is in $\ce$, if and only if $K$ is cone-injective with respect to $h$.

$(2)\Rightarrow(1)$. We adapt the proof of \cite[Theorem~7.6]{RTe2}. It is enough to prove that given any sequent
$$ \varphi(x) \vdash \bigvee_{j\in J}(\exists y_j)(\psi_j(x,y_j)\wedge \varphi(x)),$$
where $\varphi$ and $\psi$ are conjunctions of atomic formulas of
$\mathbb L_{\lambda\lambda}$, there exists a family of morphisms $(g\colon A\to B_j)_{j\in J}$ between $\lambda$-presentable $\mathbb L$-structures for which: $K$ is $\ce$-injective with respect to the $(g_j)_j$ if and only if $K\models \varphi\vdash \bigvee_{j\in J}(\exists y_j)(\psi_j\wedge\varphi)$, for any $K\in\Str(\mathbb L)$.

By \cite[Corollary~5.4]{RTe2} there exists $\lambda$-presentable $\mathbb L$-structures $A$ and $C_j$ for which $\varphi(x)$ and $\psi_j(x,y_j)$ are presentation formulas for $A$ and $C_j$ respectively (for any $j\in J$); these come together with maps $e\colon F(X)\to A$ and $e_j'\colon F(X+Y_j)\to C$. For each $j\in J$, consider now the pushout $B_j$ of $e$ along $e_j'F(i_X)$, as depicted below.
\begin{center}
	\begin{tikzpicture}[baseline=(current  bounding  box.south), scale=2]
		
		\node (a0) at (0.1,0.9) {$F(X)$};
		\node (b0) at (1.2,0.9) {$A$};
		\node (c0) at (0.1,0) {$F(X+Y_j)$};
		\node (d0) at (1.2,0) {$C_j$};
		\node (1) at (2,0.45) {$B_j$};
		
		\path[font=\scriptsize]
		
		(a0) edge [->] node [left] {$F(i_X)\ $} (c0)
		(a0) edge [->] node [above] {$e$} (b0)
		(c0) edge [->] node [below] {$e_j'$} (d0)
		(b0) edge [->] node [above] {$\ g_j$} (1)
		(d0) edge [->] node [left] {} (1);
	\end{tikzpicture}	
\end{center} 
By homming into an $\mathbb L$-structure $K$, using the definition of presentation formula, and taking the $(\ce,\cm)$ factorizations of the arrows corresponding to $g_je F(i_1)$, we obtain the pullback diagram below.
\begin{center}
	\begin{tikzpicture}[baseline=(current  bounding  box.south), scale=2]
		
		\node (0) at (-2.5,0.7) {$\Str(\LL)(B_j,K)\cong (\psi_j\wedge\varphi)_K$};
		\node (a0) at (0,1.4) {$\Str(\LL)(A,K)\cong \varphi_K$};
		\node (b0) at (2,1.4) {$K^X$};
		\node (c0) at (0.1,0) {$\Str(\LL)(C_j,K)\cong (\psi_j)_K$};
		\node (d0) at (2,0) {$K^{X+Y_j}$};
		\node (1) at (0.4,0.7) {$(\exists y_j)(\psi_j\wedge\varphi)_K$};
		
		\path[font=\scriptsize]
		
		(0) edge [->] node [above] {$\Str(\LL)(g_j,K)\ \ \ \ \ \ \ \ $} (a0)
		(0) edge [->] node [below] {} (c0)
		(a0) edge [>->] node [above] {} (b0)
		(c0) edge [>->] node [below] {} (d0)
		(d0) edge [->] node [left] {} (b0)
		
		(0) edge [->>] node [left] {} (1)
		(1) edge [>->] node [left] {} (b0);
	\end{tikzpicture}	
\end{center} 
Now, we can focus on the top square and take coproducts over $j\in J$, as well as the join of the formulas in the middle:
\begin{center}
	\begin{tikzpicture}[baseline=(current  bounding  box.south), scale=2]
		
		\node (0) at (-2.5,0.7) {$\sum_{j\in J}\Str(\LL)(B_j,K)$};
		\node (a0) at (0,1.4) {$\Str(\LL)(A,K)\cong \varphi_K$};
		\node (b0) at (2,1.4) {$K^X$};
		\node (1) at (0.4,0.7) {$\bigvee_{j\in J}(\exists y_j)(\psi_j\wedge\varphi)_K$};
		
		\path[font=\scriptsize]
		
		(0) edge [->] node [above] {$(\Str(\LL)(g_j,K))_{j\in J}\hspace{40pt}$} (a0)
		(a0) edge [>->] node [above] {} (b0)
		
		(0) edge [->>] node [left] {} (1)
		(1) edge [>->] node [left] {} (b0);
	\end{tikzpicture}	
\end{center} 

Note that by orthogonality of the factorization system we already have an inclusion $\bigvee_{j\in J}(\exists y_j)(\psi_j\wedge\varphi)_K\subseteq \varphi_K$.

Thus, $K$ is $\ce$-injective to the $(g_j)_j$'s if and only if $(\Str(\LL)(g_j,K))_{j\in J}$ is in $\ce$, if and only if $\bigvee_{j\in J}(\exists y_j)(\psi_j\wedge\varphi)_K=\varphi_K$, if and only if $K\models \varphi\vdash \bigvee_{j\in J}(\exists y_j)(\psi_j\wedge\varphi)$.
\end{proof}

\section{Cone-reflectivity and closure properties} 

In this section we further compare $\cv$-categories of models of basic theories with the notion of {\em cone $\ce$-reflectivity}, that we introduce here.

\begin{defi}
	Given a $\cv$-category $\ck$ we say that a full subcategory $J\colon \ca\hookrightarrow \ck$ is {\em cone $\ce$-reflective} if for each $K\in\ck$ there exists a family of morphisms $\{q_i\colon K\to JA_i\}_{i\in I_K}$, with $A_i\in\ca$, such that the induced morphism
	\begin{center}
		
		\begin{tikzpicture}[baseline=(current  bounding  box.south), scale=2]
			
			\node (c) at (0,0) {$\sum_{i\in I_K} \ca(A_i,-)$};
			\node (d) at (1.8,0) {$\ck(K,J-)$};

			\path[font=\scriptsize]
			
			(c) edge [->>] node [above] {} (d);
		\end{tikzpicture}
	\end{center}
	lies in $\ce$. 
\end{defi}

Below, by an accessible $\cv$-category we mean what is called {\em conically accessible} in \cite{LT23}; these are $\cv$-categories freely generated by small $\cv$-categories under $\lambda$-filtered colimits, for some $\lambda$. First we prove that closure properties imply accessibility:

\begin{rem}
	Note that we could have also considered the other (stronger) notion of accessibility of~\cite{LT23} (first introduced in~\cite{BQR}) which involves $\lambda$-flat colimits rather than $\lambda$-filtered ones. All the results below hold with this other notion of accessibility (by also replacing $\lambda$-filtered colimits with $\lambda$-flat ones everywhere); this is because the $\cv$-categories of models of basic $\mathbb L_{\lambda\lambda}$-theories can be proved to be closed under $\lambda$-flat colimits as well. Nonetheless, for most of the examples we are interested in, the two notions of accessibility coincide by~\cite{LT22}, so it makes sense to consider the more traditional notion.
\end{rem}

\begin{propo}\label{closure->accessible}
	Let $\ck$ be a locally presentable $\cv$-category. Every subcategory of $\ck$ that is closed under $\lambda$-filtered colimits and $\lambda$-pure subobjects (for some $\lambda$) is accessible. 
\end{propo}
\begin{proof}
	Let $\ca$ be any such subcategory; it follows that $\ca_0$ is accessible by \cite[Corollary~2.36]{AR}, and hence $\ca$ is accessible as a $\cv$-category by \cite[Corollary~3.23]{LT23}.
\end{proof}

\begin{propo}\label{cone-refl}
	Let $\cg$ be an $\ce$-generator of $\cv_0$. Let $\ck$ be a locally presentable $\cv$-category and $J\colon \ca\hookrightarrow\ck$ be an accessible and accessibly embedded subcategory of $\ck$ which is closed under $\cg$-powers. Then $\ca$ is cone $\ce$-reflective in $\ck$.
\end{propo}
\begin{proof}
	By \cite[Corollary~4.22]{LT23} the $\cv$-category $\ca$ is virtually reflective in $\ck$. This means that for any $K\in\ck$ the $\cv$-functor $\ck(K,J-)\colon \ca\to\cv$ is a small colimit of representables. Therefore, since weighted colimits can be computed using coequalizers, coproducts, and copowers, and regular epimorphisms are in $\ce$, if follows that for any $K\in\ck$ there are families $\{X_i\}_{i\in I_K}\subseteq\cv$ and $\{A_i\}_{i\in I_K}\subseteq \ca$ together with a map
	\begin{center}
		
		\begin{tikzpicture}[baseline=(current  bounding  box.south), scale=2]
			
			\node (c) at (0,0) {$\sum_{i\in I_K} X_i\cdot\ca(A_i,-)$};
			\node (d) at (2,0) {$\ck(K,J-)$};

			\path[font=\scriptsize]
			
			(c) edge [->>] node [above] {$q_K$} (d);
		\end{tikzpicture}
	\end{center}
	that lies in $\ce$. Since $\cg$ is an $\ce$-generator and maps in $\ce$ are stable under coproducts and copowers, without loss of generality we can assume that $X_i\in\cg$ for any $i\in I_K$. 
	
	Note now that $q_K$ corresponds, under transposition and Yoneda, to a family of maps 
	$$ (q^i_K\colon K\to X_i\pitchfork JA_i)_{i\in I_K}. $$
	Moreover, since $J$ preserves $\cg$-powers, we have that $X_i\pitchfork JA_i\cong J(X_i\pitchfork A_i)$. Then, using again Yoneda, the family $q^i_K$ corresponds to a map $\tilde q_K$ as in the triangle below
	\begin{center}
		
		\begin{tikzpicture}[baseline=(current  bounding  box.south), scale=2]
			
			\node (b) at (0,0.5) {$\sum_{i\in I_K} X_i\cdot\ca(A_i,-)$};
			\node (c) at (0,-0.5) {$\sum_{i\in I_K} \ca(X_i\pitchfork A_i,-)$};
			\node (d) at (2.5,0) {$\ck(K,J-)$};

			\path[font=\scriptsize]
			
			(b) edge [->>] node [above] {$q_K$} (d)
			(b) edge [->] node [left] {$c$} (c)
			(c) edge [->>] node [below] {$\tilde q_K$} (d);
		\end{tikzpicture}
	\end{center}
	where $c$ is the comparison map. Since $q_K\in\ce$ then also $\tilde q_K\in\ce$ and the proof is complete.
\end{proof}

\begin{propo}\label{con-inj}
	Every accessible, accessibly embedded cone $\ce$-reflective subcategory in a locally presentable $\cv$-category $\ck$ is a $\ce$-cone-injectivity class.
\end{propo}
\begin{proof}
	Let $\ca$ be a full subcategory of $\ck$ which is accessible, accessibly embedded, and cone $\ce$-reflective; denote by $J\colon\ca\hookrightarrow\ck$ the inclusion.
	
	Fix $\lambda$ such that $\ca$, $\ck$, and $J$ are $\lambda$-accessible and $J$ preserves $\lambda$-presentable objects. For any $K\in\ck_\lambda$ consider a family $\{q_i\colon K\to JA_i\}_{i\in I_K}$ expressing the cone $\ce$-reflectivity of $\ca$. Since $\ca_0$ is also accessible and accessibly embedded in $\ck_0$, it follows that it is cone reflective in the ordinary sense; thus for any $K\in\ck_\lambda$ we also have a family $\{q_j\colon K\to JA_j\}_{i\in J_K}$ expressing the cone-reflectivity of $\ca_0$.
	
	It follows that if we take the family $\{q_l\colon K\to JA_l\}_{l\in J_K\cup I_K}$ then for any $A\in\ca$ the map 
	$$ 
	q_K\colon \sum\limits_{l\in I_K\cup J_K} \ca(A_l,A)\longrightarrow\ck(K,A)
	$$
	is in $\ce$ and is a surjection, since we have the following commutative triangles
	\begin{center}
		\begin{tikzpicture}[baseline=(current  bounding  box.south), scale=2]
			
			\node (a) at (0,0.8) {$\sum_{i\in I_K} \ca(A_i,A)$};
			\node (b) at (2.2,0.8) {$\sum_{l\in I_K\cup J_K} \ca(A_l,A)$};
			\node (c) at (4.4,0.8) {$\sum_{j\in J_K} \ca(A_j,A)$};
			\node (d) at (2.2,0) {$\ck(K,JA)$};

			\path[font=\scriptsize]
			
			(a) edge [->] node [above] {} (b)
			(c) edge [->] node [left] {} (b)
			(a) edge [->] node [below] {$\ce$} (d)
			(b) edge [->] node [right] {$q_K$} (d)
			(c) edge [->] node [below] {$\tx{surj.}$} (d);
		\end{tikzpicture}
	\end{center}
	where the map on the left is in $\ce$ and the one on the right is a surjection.
	
	Without loss of generality we can assume that each $A_l$ lies in $\ca_\lambda$. Indeed, for any $l\in L_K:=I_K\cup J_K$ the map $q_l\colon K\to JA_l$ factors through some $q_l'\colon K\to JA_l'$, with $A_l'\in\ca_\lambda$ since this generates $\ca$ under $\lambda$-filtered colimits, $J$ preserves them, and $K$ is $\lambda$-presentable in $\ck$. Then we have a commutative diagram
	\begin{center}
		
		\begin{tikzpicture}[baseline=(current  bounding  box.south), scale=2]
			
			\node (b) at (0,0.45) {$\sum_{l\in L_K} \ca(A_l,-)$};
			\node (c) at (0,-0.45) {$\sum_{l\in L_K} \ca(A'_l,-)$};
			\node (d) at (2.3,0) {$\ck(K,J-)$};

			\path[font=\scriptsize]
			
			(b) edge [->>] node [above] {$q_K$} (d)
			(b) edge [->] node [left] {} (c)
			(c) edge [->>] node [below] {$q'_K$} (d);
		\end{tikzpicture}
	\end{center}
	where $q_X'$ is induced by the $q_l'$, and is in $\ce$ and a surjection since $q_X$ was.
	
	Therefore we assume $A_i\in\ca_\lambda$ for all $i\in I_K$. Now consider the set of cones 
	$$ \ch:=\{ (q_i\colon K\to JA_i)_{i\in I_K}\ | \ K\in\ck_\lambda \}; $$ 
	we shall prove that $\ca=\tx{Inj}_\ce(\ch)$.
	
	The inclusion $\ca\subseteq \tx{Inj}_\ce(\ch)$ is trivial. Consider then $Y\in\tx{Inj}_\ce(\ch)$; to conclude it is enough to prove that the inclusion of $J\ca_\lambda/Y$ into $\ck_\lambda/Y$ is final. Indeed, then $J\ca_\lambda/Y$ would be $\lambda$-filtered and would have colimit $Y$, showing that $Y\in\ca$.
	
	To prove that the inclusion is final it suffices to show that any morphism $K\to Y$, with $K\in\ck_\lambda$, factors through some $JA$ with $A\in\ca_\lambda$ (this is enough to imply finality since $\ck_\lambda/Y$ is $\lambda$-filtered). Now, by hypothesis on $Y$ we have that the induced map
	$$ \textstyle\sum_l \ck(JA_l,Y)\longrightarrow\ck(K,Y) $$ 
	is in particular a surjection; thus the desired factorization exists.
\end{proof}

\begin{coro}\label{char0}
Let $\cv_0$ have an $\ce$-generator $\cg$ made of $\ce$-stable $\cv$-connected objects. Then every category $\Mod(\mathbb T)$
of models of a basic $\mathbb L_{\lambda\lambda}$-theory is a 
$(\mu,\ce)$-cone-injectivity class of $\Str(\mathbb L)$ for some $\mu\geq\lambda$.
\end{coro} 
\begin{proof}
By Propositions~\ref{examples-lambda-p.e.elementary} and~\ref{stable1} the $\cv$-category $\Mod(\mathbb T)$ satisfies the hypotheses of \ref{closure->accessible}; then apply \ref{con-inj} and \ref{cone-refl}.
\end{proof}

We can collect all the results together in the following Theorem.

\begin{theo}\label{char}
Assume that $\cv_0$ has an $\ce$-generator consisting of $\ce$-stable $\cv$-connected objects. Consider the following conditions for a full subcategory $\ca$ of $\Str(\mathbb L)$:
\begin{enumerate}
\item $\ca$ is isomorphic to $\Mod(\mathbb T)$ of models of a basic $\mathbb L$-theory $\mathbb T$;
\item $\ca$ is closed under $\lambda$-filtered colimits, powers by $\ce$-stable $\cv$-connected objects, and $\lambda$-pure subobjects, for some $\lambda$;
\item $\ca$ is accessible, accessibly embedded, and closed under powers by $\ce$-stable $\cv$-connected objects;
\item $\ca$ is a $(\lambda,\ce)$-cone-injectivity class of $\Str(\mathbb L)$, for some $\lambda$.
\end{enumerate}
Then $(1)\Rightarrow(2)\Rightarrow(3)\Rightarrow(4)$ always hold, and $(4)\Rightarrow(1)$ holds under Assumption~\ref{assumption}.
\end{theo}
\begin{proof}
The implication (1)$\Rightarrow$(2) is a consequence of Propositions~\ref{examples-lambda-p.e.elementary} and~\ref{stable1}, (2)$\Rightarrow$(3) is given by Proposition~\ref{closure->accessible}, (3)$\Rightarrow$(4) by Propositions~\ref{cone-refl} and~\ref{con-inj} together, and finally (4)$\Rightarrow$(1) is given by Proposition~\ref{mod<->inj}.
\end{proof}

Note that the equivalence (3)$\Leftrightarrow$(2) for a fixed $\lambda$ is not valid even in the ordinary context, as shown in~\cite[Example~3.1]{RAB02}. 
The base $\Met$ satisfies the assumption of \ref{char} for $\ce$ = dense, but we will show that the implication (3)$\Rightarrow$(2), for a fixed $\lambda$, is not valid.

\begin{defi}
We say that a morphism $f:K\to L$ in a $\cv$-category $\ck$ is a
\textit{$(\lambda,\ce)$-pure quotient} if $\ck(A,K)\to\ck(A,L)$ is in $\ce$
for every $\lambda$-presentable object $A$ in $\ck$.

A class $\cl\subseteq\ck$ is \textit{closed under $(\lambda,\ce)$-pure quotients} if, for a $(\lambda,\ce)$-pure quotient $f\colon K\to L$, then $L\in\cl$ provided that $K\in\cl$. 
\end{defi}

If $\ce$ is made from epimorphisms then $(\lambda,\ce)$-pure quotients
are $\lambda$-pure quotients in the sense of \cite{AR2}.

\begin{lemma}\label{quot}
Every $(\lambda,\ce)$-cone-injectivity class of $\mathbb L$-structures
is closed under $(\lambda,\ce)$-pure quotients.
\end{lemma}
\begin{proof}
Let $h=(h_j\colon A\to B_j)_{j\in J}$ be a cone between $\lambda$-presentable $\mathbb L$-structures, $K$ be $h$-injective, and $f:K\to L$ be a $(\lambda,\ce)$-pure quotient. Consider the commutative diagram
\begin{center}
			\begin{tikzpicture}[baseline=(current  bounding  box.south), scale=2]
				
				\node (0) at (0.5,0.8) {$\sum\ck(B_j,K)$};
				\node (a0) at (2.2,0.8) {$\ck(A,K) $};
				\node (b0) at (2.2,0) {$\ck(A,L) $};
				\node (d0) at (0.5,0) {$\sum\ck(B_j,L)$};
				
				\path[font=\scriptsize]
				
				(0) edge [->>] node [above] {} (a0)
				(a0) edge [->>] node [above] {} (b0)
				(0) edge [->>] node [left] {} (d0)
				(b0) edge [<-] node [below] {} (d0);
			\end{tikzpicture}	
		\end{center}
where the horizontal maps are induced by $h$ and the vertical maps 
by $f$.	Since all maps, except the lower horizontal one, are in $\ce$,
the lower horizontal map is in $\ce$ as well. Hence $L$ is 
$h$-injective.	
\end{proof}

\begin{exam}
Let $\cv=\Met$ and $\ce=$ dense. Consider the language $\mathbb L$
consisting of unary relation symbols $R_n$, $n=1,2,\dots,n,\dots$.
Let $1$ be a one-point metric space and $A_n$ a one-point metric space
where $(R_n)_{A_n}=A_n$ and $(R_n)_{A_m}=\emptyset$ for $n\neq m$.
Let $h_n:1\to A_n$, $n=1,2,\dots,n,\dots$ be a cone $h$. Let $\cl$
consist of objects injective to $h$ in an unenriched sense. This means
that $A\in\cl$ if and only if every $a\in A$ is in some relation $R_n$. 
Then $\cl$ is closed in $\Str(\mathbb L)$ under directed colimits and submodels, and therefore under $\lambda$-p.e.\! elementary subobjects.
%. Since the factorization system (dense, closed isometry) 
%satisfies \ref{assumption}, $(\lambda,\ce)$-pure subobjects are
%$\lambda$-elementary (\cite[6.3]{RTe2}), hence submodels. Thus $\cl$
%is closed under $(\lambda,\ce)$-pure subobjects. 

Finally, $\ce$-stable objects are precisely discrete objects. Indeed, for a metric space $X$, we take metric spaces $X_n$,
$n=1,2,\dots,n,\dots$, given by adding to $X$ new points $x_n$ with $d(x,x_n)=\frac{1}{n}$. The distances $d(x_n,y_n)$ are given by the triangle inequalities as $d(x_n,y_n)=d(x,y)+\frac{2}{n}$. Let
$X_\ast=\bigcup_n X_n$. Then $e:X_\ast\to X_\ast\cup X$ is dense but
$e^X$ is dense only for $X$ discrete. Hence $1$ is the only $\ce$-stable
connected object and thus $\cl$ is closed under powers by $\ce$-stable
connected objects. Consequently, $\cl$ satisfies the properties from
\ref{stable1}.

However, $\cl$ is not closed under $(\lambda,\ce)$-pure quotients. It
suffices to take the metric space $A=\{0,1,\frac{1}{2},\dots,\frac{1}{n},\dots\}$ where
$(R_n)_A=\{\frac{1}{n}\}$. Then the inclusion $A\setminus\{0\}\to A$
is a $(\lambda,\ce)$-pure quotient, $A\setminus\{0\}\in\cl$ but $A\notin\cl$. Following \ref{quot}, $\cl$ is not an enriched cone
injectivity class.
\end{exam}

\begin{rem}
This example shows that (3) in \ref{char} does not imply that $\ca$ is closed under $(\lambda,\ce)$-pure quotients. Moreover, example 
\cite[3.1]{RAB02} is closed under $(\omega,\ce)$-pure quotients
for $\ce=$ surjections. Hence, even by adding to be closed under
$(\lambda,\ce)$-pure quotients to (3) of \ref{char}, we do not obtain
a characterization of $\lambda$-cone-injectivity classes for a fixed
$\lambda$.
\end{rem}
%\begin{pb}
%Can $\lambda$-cone-injectivity classes in $\Met$ (for $\ce$ = dense)
%be characterized as classes closed under $\lambda$-directed colimits, $(\lambda,\ce)$-pure subobjects and $\ce$-pure quotients?
%\end{pb}

\appendix

\section{When elementary implies p.e.\! elementary}\label{elementary-section}

In the two lemmas below we give conditions so that $\lambda$-elementary morphisms are $\lambda$-p.e.\! elementary:

\begin{lemma}\label{pos-ex}
	Assume that $\cv$ has an $\ce$-generator consisting of $\ce$-projective objects and that coproducts in $\cv$ are universal. Then $\lambda$-elementary morphisms are $\lambda$-p.e.\! elementary in any language $\LL$.
\end{lemma}
\begin{proof}
	Let $f\colon K\to L$ be $\lambda$-elementary and 
	$$\psi(x)=\bigvee_{j\in J}\psi_j$$ where $\psi_j$, $j\in J$ are positive-primitive. Following \cite[6.3]{RTe2},
	\begin{center}
		\begin{tikzpicture}[baseline=(current  bounding  box.south), scale=2]
			
			\node (a0) at (0,0.8) {$(\psi_j)_K$};
			\node (a0') at (0.2,0.6) {$\lrcorner$};
			\node (b0) at (1,0.8) {$(\psi_j)_L$};
			\node (c0) at (0,0) {$K^X$};
			\node (d0) at (1,0) {$L^X$};
			
			\path[font=\scriptsize]
			
			(a0) edge [->] node [above] {} (b0)
			(a0) edge [>->] node [left] {} (c0)
			(b0) edge [>->] node [right] {} (d0)
			(c0) edge [->] node [below] {$f^X$} (d0);
		\end{tikzpicture}	
	\end{center}
	are pullbacks. Since coproducts are universal, the following
	\begin{center}
		\begin{tikzpicture}[baseline=(current  bounding  box.south), scale=2]
			
			\node (a0) at (0,0.8) {$\sum_{j\in J}(\psi_j)_K$};
			\node (a0') at (0.2,0.6) {$\lrcorner$};
			\node (b0) at (1.7,0.8) {$\sum_{j\in J}(\psi_j)L$};
			\node (c0) at (0,0) {$K^X$};
			\node (d0) at (1.7,0) {$L^X$};
			
			\path[font=\scriptsize]
			
			(a0) edge [->] node [above] {} (b0)
			(a0) edge [>->] node [left] {} (c0)
			(b0) edge [>->] node [right] {} (d0)
			(c0) edge [->] node [below] {$f^X$} (d0);
		\end{tikzpicture}	
	\end{center}
	is a pullback. To conclude we have to prove that also
	\begin{center}
		\begin{tikzpicture}[baseline=(current  bounding  box.south), scale=2]
			
			\node (a0) at (0,0.8) {$\bigvee_{j\in J}(\psi_j)_K$};
			%\node (a0') at (0.2,0.6) {$\lrcorner$};
			\node (b0) at (1.7,0.8) {$\bigvee_{j\in J}(\psi_j)_L$};
			\node (c0) at (0,0) {$K^X$};
			\node (d0) at (1.7,0) {$L^X$};
			
			\path[font=\scriptsize]
			
			(a0) edge [->] node [above] {} (b0)
			(a0) edge [>->] node [left] {} (c0)
			(b0) edge [>->] node [right] {} (d0)
			(c0) edge [->] node [below] {$f^X$} (d0);
		\end{tikzpicture}	
	\end{center}
	is a pullback. Consider a commutative diagram
	\begin{center}
		\begin{tikzpicture}[baseline=(current  bounding  box.south), scale=2]
			
			\node (a0) at (0,0.8) {$X$};
			\node (b0) at (1.7,0.8) {$\bigvee_{j\in J}(\psi_j)_L$};
			\node (c0) at (0,0) {$K^X$};
			\node (d0) at (1.7,0) {$L^X$};
			
			\path[font=\scriptsize]
			
			(a0) edge [->] node [above] {} (b0)
			(a0) edge [>->] node [left] {} (c0)
			(b0) edge [>->] node [right] {} (d0)
			(c0) edge [->] node [below] {$f^X$} (d0);
		\end{tikzpicture}	
	\end{center}
	Since $\cv$ has an $\ce$-generator of $\ce$-projective objects, there
	exists $e:Y\to X$ in $\ce$ where $Y$ is $\ce$-projective. Hence the composite $Y\to\bigvee_{j\in J}(\psi_j)_L$ lifts to 
	$\sum_{j\in J}(\psi_j)_L$. Let $Y\to \sum_{j\in J}(\psi_j)_K$ be the morphism induced by the universal property of the first pullback, and $t\colon Y\to (\bigvee_{j\in J}\psi_j)_L$ its composite with $\sum_{j\in J}(\psi_j)_L\twoheadrightarrow (\bigvee_{j\in J}(\psi_j)_L$.
	We then obtain the commutative square
	\begin{center}
		\begin{tikzpicture}[baseline=(current  bounding  box.south), scale=2]
			
			\node (a0) at (0,0.8) {$Y$};
			\node (b0) at (1,0.8) {$(\bigvee_{j\in J}\psi_j)_L$};
			\node (c0) at (0,0) {$X$};
			\node (d0) at (1,0) {$K^X$};
			
			\path[font=\scriptsize]
			
			(a0) edge [->] node [above] {$t$} (b0)
			(a0) edge [->>] node [left] {$e$} (c0)
			(b0) edge [>->] node [right] {} (d0)
			(c0) edge [->] node [below] {} (d0);
		\end{tikzpicture}	
	\end{center}
	Then the diagonal $X\to\bigvee_{j\in J}(\psi_j)_L$, induced by orthogonality of the factorization system, is the map needed to make our diagram a pullback.
\end{proof}

We give other conditions below, examples of such $\cv$'s are $\Ab$, $\bo{CMon}$, $R$-$\Mod$, $\bo{GAb}$, and $\bo{DGAb}$ with the (regular epi, mono) factorization system.

\begin{lemma}\label{pos-ex-additive}
	Let $\cv$ be locally finitely presentable, have finite direct sums (that is, $\cv$ is $\bo{CMon}$-enriched), and such that $\ce$ is pullback stable. Then $\lambda$-elementary morphisms are $\lambda$-p.e.\! elementary in any language $\LL$.
\end{lemma}
\begin{proof}
	Since $\cv$ has finite direct sums, given any language $\LL$ and arity $X\in\cv$ we have a sum function symbol $+:(X\oplus X,X)$. We shall first prove that for any $X$-ary formulas $\varphi$ and $\psi$ we have
	$$ (\varphi\vee\psi)(x)\equiv (\exists (y,z):X+X)\  (\varphi(y)\wedge\psi(z)\wedge (y+z\eqcirc x)), $$
	meaning that they have the same interpretation under any $\LL$-structure. This way it follows that finite disjunctions can be captured under the regular fragment.
	
	Given an $\LL$-structure $A$, call $\chi$ the formula on the right hand side; then $\chi$ is interpreted as follows
	\begin{center}
		\begin{tikzpicture}[baseline=(current  bounding  box.south), scale=2]
			
			\node (a0) at (0,0.8) {$(\varphi(y)\wedge\psi(z))_A$};
			\node (a0') at (0.2,0.55) {$\lrcorner$};
			\node (b0) at (1.5,0.8) {$A^X\oplus \psi_A$};
			\node (c0) at (0,0) {$\varphi_A\oplus A^X$};
			\node (d0) at (1.5,0) {$A^X\oplus A^X$};
			\node (e0) at (2.8,0) {$A^X$};
			\node (f0) at (2.4,1) {$\chi_A$};
			
			\path[font=\scriptsize]
			
			(a0) edge [>->] node [above] {} (b0)
			(a0) edge [>->] node [left] {} (c0)
			(b0) edge [>->] node [right] {} (d0)
			(c0) edge [>->] node [below] {} (d0)
			(d0) edge [->] node [above] {$+$} (e0)
			(a0) edge [bend left=15, ->>] node [above] {$\ce$} (f0)
			(f0) edge [bend left=15, >->] node [right] {$\cm$} (e0);
		\end{tikzpicture}	
	\end{center} 
	but $(\varphi(y)\wedge\psi(z))_A\cong \varphi_A\oplus\psi_A$; thus $\chi_A$ coincides by definition with $(\varphi\vee\psi)_A$.
	
	Consider now a $\lambda$-elementary morphism $f\colon K\to L$. It follows that $f$ is elementary with respect to any $\psi$ that is the finite disjunction of positive-primitive formulas. Indeed, any such $\psi$ is equivalent to a positive-primitive formula (by pullback stability of $\ce$); thus the square in the definition of elementary morphism is a pullback by definition.\\
	If $\psi=\bigvee_{j\in J}\phi_j$ is an infinite disjunction of positive primitive formulas; then the interpretation of $\psi$ in an $\LL$-structure can be seen as the filtered colimit of the interpretations of $\bigvee_{j\in J'}\phi_j$ where $J'\subseteq J$ is finite. Since filtered colimits are stable under pullbacks in a locally finitely presentable category, it follows that also the interpretation of $\psi$ is preserved under the pullback.
\end{proof}

We have shown in \cite{RTe2} that $\lambda$-elementary morphisms in $\Str(\LL)$ are deeply connected to the notion of {\em $(\lambda,\ce)$-pure morphism} introduced in \cite{RTe1}. Indeed, every $(\lambda,\ce)$-pure morphism is always $\lambda$-elementary, and the converse holds under Assumption~\ref{assumption}. This is shown as \cite[Proposition~6.6]{RTe2} where both directions are stated under the additional hypotheses of \ref{assumption}, but only one actually uses them.

In this context we can describe $\lambda$-p.e.\! elementary morphisms as follows. Note that we are not making use anywhere of the actual notion of $(\lambda,\ce)$-pure morphism; because of that, and of the technical nature of the definition, we avoid recalling it here.

\begin{lemma}\label{lambda-stable-categorically}
	Under Assumption~\ref{assumption}, $f\colon K\to L$ is a $\lambda$-p.e.\! elementary morphisms in $\Str(\LL)$ if and only if it is $(\lambda,\ce)$-pure and the following property holds: for any $\lambda$-small cone $(g_j\colon A\to B_j)_{j\in J}$, with $A,B_j\in\Str(\LL)_\lambda$, the bottom square below
	\begin{center}
		\begin{tikzpicture}[baseline=(current  bounding  box.south), scale=2]
			
			\node (a0) at (0,1.6) {$\sum_{j\in J}\Str(\LL)(B_j,K)$};
			\node (b0) at (2.5,1.6) {$\sum_{j\in J}\Str(\LL)(B_j,L)$};
			\node (a1) at (0,0.8) {$\bigvee_{j\in J}\Str(\LL)(B_j,K)$};
			\node (b1) at (2.5,0.8) {$\bigvee_{j\in J}\Str(\LL)(B_j,L)$};
			\node (a0') at (0.2,0.55) {$\lrcorner$};
			\node (a2) at (0,0) {$\Str(\LL)(A,K)$};
			\node (b2) at (2.5,0) {$\Str(\LL)(A,L)$};
			
			\path[font=\scriptsize]
			
			(a0) edge [->] node [above] {} (b0)
			(a1) edge [->] node [above] {} (b1)
			(a2) edge [->] node [below] {$\Str(\LL)(A,f)$} (b2)
			
			(a0) edge [->>] node [left] {$\ce$} (a1)
			(a1) edge [>->] node [left] {$\cm$} (a2)
			(b0) edge [->>] node [right] {$\ce$} (b1)
			(b1) edge [>->] node [right] {$\cm$} (b2);
		\end{tikzpicture}	
	\end{center}
	is a pullback in $\cv$; here the objects in the middle are the $(\ce,\cm)$ factorizations of the maps induced by precomposition with the $g_j$'s.
\end{lemma}
\begin{proof}
	As we observed above, $\lambda$-elementary and $(\lambda,\ce)$-pure morphisms coincide in $\Str(\LL)$. Moreover, note that the object $\bigvee_{j\in J}\Str(\LL)(B_j,L)$ can be equivalently defined as follows: first take the $(\ce,\cm)$-factorization $\cp(g_j,L)$ of each $\Str(\LL)(g_j,L)$, then $\bigvee_{j\in J}\Str(\LL)(B_j,L)$ coincides with the $(\ce,\cm)$-factorization of the induced map 
	$$  \sum_{j\in J}\cp(g_j,L)\longrightarrow \Str(\LL)(A,L). $$

	Fix now any $\lambda$-p.e.\! elementary morphism $f\colon K\to L$, then $f$ is $(\lambda,\ce)$-pure (being $\lambda$-elementary). Given $(g_j\colon A\to B_j)_{j\in J}$ as above, by~\cite[Lemma~5.11]{RTe2} for any $j\in J$ we can find a positive primitive formula $\psi_j$ for which $(\psi_j)_L$ coincides with the $(\ce,\cm)$-factorization of $\Str(\LL)(g_j,L)$ for any $L\in\Str(\LL)$; that is
	$$ (\psi_j)_L\cong \cp(g_j,L). $$
	By the observation above and how disjunction is interpreted it follows that 
	$$ (\bigvee_{j\in J}\psi_j)_L\cong \bigvee_{j\in J}\Str(\LL)(B_j,L)$$ 
	for any $L\in\Str(\LL)$. Thus, the fact that the required square is a pullback follows by $f$ being $\lambda$-p.e. elementary.
	
	Assume conversely that $f\colon K\to L$ is $(\lambda,\ce)$-pure and satisfies the condition of the Lemma. Then it is $\lambda$-elementary. Consider a disjunction $\bigvee_{j\in J}\psi_j$ of positive-primitive formulas in $\LL_{\lambda\lambda}$, of arity $X\in\cv_\lambda$. Then for any $j\in J$ we have that
	$$\psi_j(x):= (\exists y)\varphi_j(x,y)$$
	where $\varphi_j$ is a $\lambda$-ary conjunction of atomic formulas. By~\cite[Corollary~5.4]{RTe2} for any $j\in J$ there is a $\lambda$-presentable object $B_j\in\Str(\LL)$ such that $\varphi_j$ is a presentation formula for $B_j$; that is $(\varphi_j)_{(-)}\cong \Str(\LL)(A_j,-)$. Under this identification, the $\cv$-natural transformation sending $L\in\Str(\LL)$ to 
	$$ \Str(\LL)(B_j,-)\cong(\varphi_j)_{L}\rightarrowtail L^{X+Y_j}\xrightarrow{\pi_1} L^X\cong\Str(\LL)(FX,L) $$
	corresponds, thanks to Yoneda, to a morphism $g_j\colon FX\to B_j$. In particular, following the notation above, it follows that
	$$ (\psi_j)_L\cong \cp(g_j,L). $$
	Thus we also have that 
	$$ (\bigvee_{j\in J}\psi_j)_L\cong \bigvee_{j\in J}\Str(\LL)(B_j,L).$$
	Thus being elementary with respect to $\bigvee_{j\in J}\psi_j$ can be seen as a consequence of the square in the statement being a pullback.
\end{proof}

\begin{exam}
	Let us consider $\cv=\Met$ and $\ce$ consisting of the dense morphisms of metric spaces. We show that $\lambda$-elementary morphisms are $\lambda$-p.e.\ elementary under Assumption~\ref{assumption}.
	Note that neither \ref{pos-ex} nor \ref{pos-ex-additive} apply to this case, but we can still use the lemma above.
	
	We shall prove that the condition of Lemma~\ref{lambda-stable-categorically} holds more generally when $\Str(\LL)$ is replaced by any locally $\lambda$-presentable $\Met$-category. Consider a locally $\lambda$-presentable $\Met$-category $\ck$ and a cone
	$(g_j:A\to B_j)_{j\in J}$ of morphisms in $\ck$ with $\lambda$-presentable domain and codomain. 
	
	We need to show that for any $f\colon K\to L$ that is $(\lambda,\ce)$-pure, the square
	\begin{center}
		\begin{tikzpicture}[baseline=(current  bounding  box.south), scale=2]
			
			\node (a1) at (0,0.8) {$\bigvee_{j\in J}\ck(B_j,K)$};
			\node (b1) at (2.5,0.8) {$\bigvee_{j\in J}\ck(B_j,L)$};
			%\node (a0') at (0.2,0.55) {$\lrcorner$};
			\node (a2) at (0,0) {$\ck(A,K)$};
			\node (b2) at (2.5,0) {$\ck(A,L)$};
			
			\path[font=\scriptsize]
			
			(a1) edge [->] node [above] {} (b1)
			(a2) edge [->] node [below] {$\ck(A,f)$} (b2)
			
			(a1) edge [>->] node [left] {$\cm$} (a2)
			(b1) edge [>->] node [right] {$\cm$} (b2);
		\end{tikzpicture}	
	\end{center}
	is a pullback. An element of the pullback is a morphism $u\colon A\to K$ such that for any $\epsilon>0$ there exists $j\in J$ and $v_j\colon B_j\to L$ for which the square
	\begin{center}
		\begin{tikzpicture}[baseline=(current  bounding  box.south), scale=2]
			
			\node (a0) at (0,0.8) {$A$};
			\node (b0) at (1,0.8) {$B_j$};
			\node (c0) at (0,0) {$K$};
			\node (d0) at (1,0) {$L$};
			
			\path[font=\scriptsize]
			
			(a0) edge [->] node [above] {$g_j$} (b0)
			(a0) edge [->] node [left] {$u$} (c0)
			(b0) edge [->] node [right] {$v_j$} (d0)
			(c0) edge [->] node [below] {$f$} (d0);
		\end{tikzpicture}	
	\end{center}
	commutes up to $\epsilon$. Since $f$ is a weakly $\lambda$-pure morphism by \cite[Proposition~6.9]{RTe1}, it follows that, in the setting above, there exists also $t_j\colon B_j\to K$ such that $ u \sim_{2\epsilon} t_j \circ g_j$. Since this holds for arbitrarily small $\epsilon$, then it follows that $u\in \bigvee_{j\in J}\ck(B_j,K)$, showing that the considered square is a pullback.
\end{exam}

\section{Limit theories}\label{limit-section}

Here we introduce {\em limit theories} and their models; the objective is to characterize those full subcategories of $\Str(\LL)$ that arise as orthogonality classes. These arise as an enriched generalisation of those considered in~\cite[Definition~5.7]{AR}; in other references their finitary version is also referred to as {\em cartesian theory}.

 As we shall see below limit theories will not generally fall into the fragment of regular logic, so they require an ad hoc definition; this is because ``being in $\cm$'' may not be expressible by a regular sentence.
 At the end of the section we treat the case where they can indeed be captured in the regular fragment

\begin{defi}\label{limit}
	Given conjunctions of atomic formulas $\varphi(x)$ and $\psi(x,y)$, we say that an $\mathbb L$-structure $A$ satisfies the sequent
	$$ \varphi(x) \vdash (\exists!! y) (\psi(x,y)\wedge \varphi(x)),$$ and write
	$ A\models (\varphi\vdash \exists!!y(\psi\wedge\varphi))$, if we have:\begin{itemize}
		\item(existence) $A\models (\varphi\vdash \exists y(\psi\wedge \varphi))$;
		\item($\cm$-uniqueness) The diagonal map
		\begin{center}
			\begin{tikzpicture}[baseline=(current  bounding  box.south), scale=2]
				
				\node (0) at (0.5,0.8) {$(\psi\wedge\varphi)_A$};
				%\node (a0) at (2.2,0.8) {$(\exists y(\psi\wedge\varphi))_A$};
				\node (b0) at (2.2,0) {$A^X$};
				\node (d0) at (0.5,0) {$A^{X+Y}$};
				
				\path[font=\scriptsize]
				
				(0) edge [->] node [above] {} (b0)
				(0) edge [>->] node [left] {} (d0)
				(b0) edge [<-] node [below] {$A^{i_X}$} (d0);
			\end{tikzpicture}	
		\end{center}
		lies in $\cm$.
	\end{itemize}
	
\end{defi}

We define limit theories as follows:

\begin{defi}
	A {\em limit theory} $\mathbb T$ is a set of sequents of the form
	$$  \varphi(x) \vdash (\exists!!y)(\psi(x,y)\wedge\varphi(x)$$
	where $\varphi$ and $\psi$ are conjunctions of atomic formulas.
\end{defi}

\begin{rem}
	By taking $\varphi:= \top$ in the sequent above, we are allowed to consider axioms of the form 
	$$ (\forall x)(\exists!! y) \psi(x,y). $$
	An $\LL$-structure satisfies this if and only if the composite 
	$$ \psi_A\rightarrowtail A^{X+Y}\xrightarrow{\pi_1} A^X $$
	is an isomorphism. 
\end{rem}

We can reinterpret the definition in a more diagrammatic sense as follows:

\begin{lemma}\label{pullback-validity}
	An $\LL$-structure $A$ is such that $ A\models (\varphi\vdash \exists!!y(\psi\wedge \varphi))$ if and only if there exists an arrow $t\colon\varphi_A\to \psi_A$ making the diagram
	\begin{center}
		\begin{tikzpicture}[baseline=(current  bounding  box.south), scale=2]
			
			\node (0) at (0,0.8) {$\varphi_A$};
			\node (a0) at (2,0.8) {$\varphi_A$};
			\node (b0) at (2,0) {$A^X$};
			\node (c0) at (0,0) {$\psi_A$};
			\node (d0) at (1,0) {$A^{X+Y}$};
			
			\path[font=\scriptsize]
			
			(0) edge [->] node [above] {$1$} (a0)
			(a0) edge [>->] node [above] {} (b0)
			(0) edge [dashed,->] node [left] {$t$} (c0)
			(b0) edge [<-] node [below] {$A^{i_X}$} (d0)
			(c0) edge [>->] node [below] {} (d0);
		\end{tikzpicture}	
	\end{center}
	a pullback, where $i_X\colon X\to X+Y$ is the first component inclusion.
\end{lemma}
\begin{proof}
	Consider the following diagram
	\begin{center}
		\begin{tikzpicture}[baseline=(current  bounding  box.south), scale=2]
			
			\node (0) at (0,1) {$(\psi\wedge\varphi)_A$};
			\node (a0) at (3,1) {$\varphi_A$};
			\node (b0) at (3,0) {$A^X$};
			\node (c0) at (0,0) {$\psi_A$};
			\node (d0) at (1.5,0) {$A^{X+Y}$};
			\node (e0) at (1.5,0.5) {$(\exists y(\psi\wedge\varphi))_A$};
			
			\path[font=\scriptsize]
			
			(0) edge [->] node [above] {$s$} (a0)
			(a0) edge [>->] node [above] {} (b0)
			(0) edge [->] node [left] {$t$} (c0)
			(b0) edge [<-] node [below] {$A^{i_X}$} (d0)
			(c0) edge [>->] node [below] {} (d0)
			(0) edge [->>] node [below] {$q$} (e0)
			(e0) edge [>->] node [above] {} (b0)
			(e0) edge [>->] node [below] {$m$} (a0);
		\end{tikzpicture}	
	\end{center}
	whose delimiting square is a pullback by definition, and where we noted that 
$$
(\exists y(\psi\wedge \varphi))_A\subseteq \varphi_A
$$ 
as $\cm$-subobjects of $A^X$.
	
	Now, if $ A\models (\varphi\vdash \exists!!\psi)$ then we also have $\varphi_A\subseteq (\exists y(\psi\wedge \varphi))_A$ so that $m$ is an isomorphism. Moreover, the diagonal of the square is in $\cm$ by definition, thus $q$ is an isomorphism too (being in $\ce$), and thus so is $s\colon (\psi\wedge\varphi)_A\to \varphi_A$. Therefore $(\psi\wedge\varphi)_A\cong\varphi_A$ and the statement follows.
	
	Conversely, if $\varphi_A$ is a pullback as in the statement then the map $s$ in the diagram above is an isomorphism, and so is $m$. If follows that the diagonal of the square is in $\cm$ and that $\varphi_A\subseteq (\exists y(\psi\wedge \varphi))_A$; thus $A\models (\varphi\vdash \exists!!y\psi)$.
\end{proof}

The result above shows in particular that satisfaction of a limit sentence does {\em not} depend on the chosen factorization system on $\cv$; for this reason in the following example it's not really necessary to specify it.

\begin{exam}
	Consider $\cv=\Met$. A metric space $M$ is said to have the {\em unique midpoint property} if for any $x,y\in M$ there exists a unique $z\in M$ with $d(x,z)=d(y,z)=d(x,y)/2$. This can be expressed by a limit theory $\TT$, over the empty language, with axioms
	$$ (\forall (x,y):2_d)(\exists!! \bar z:3_{d/2})\  (\pi_1(\bar z)\eqcirc x )\wedge (\pi_3(\bar z)\eqcirc y). $$
	for any $d\in\mathbb Q^+$, where $2_d$ has two elements of distance $d$, while $3_{d/2}$ has three points $(x,z,y)$ with $d(x,z)=d(y,z)=d/2$ and $d(x,y)=d$. 
\end{exam}

\begin{exam}
	For any $\cv$, we describe a single-sorted limit theory for categories internal to $\cv$. Consider the language $\LL$ with:\begin{itemize}
		\item two $(I,I)$-ary function symbols $\tx{s},\tx{t}$ --- to think of as the source and the target;
		\item a $I$-ary relation symbol $\tx{Obj}$;
		\item a $(I+I+I)$-ary relation symbol $\tx{Comp}$.
	\end{itemize}
	The idea is that the underlying object of an $\LL$-structure $A$ is the object of morphisms of an internal category, $\tx{Obj}_A$ is the object of objects (seen as the subobject of $A$ given by the identity morphisms), $s_A,t_A$ are the source and target maps, and $\tx{Comp}_A$ is the object of triples $(f,g,h)$ where $f$ and $g$ are composable and $g\circ f=h$.\\
	With this in mind, we define the limit theory $\TT$ to consist of the following axioms:\begin{enumerate}
		\item $ ( x)\ \tx{Obj}(s(x))\wedge \tx{Obj}(t(x)); $
		\item $ ( x,y,z)\ \tx{Comp}(x,y,z)\vdash t(x)\eqcirc s(y); $
		\item $ ( x,y)\ t(x)=s(y) \vdash (\exists!! z) \tx{Comp}(x,y,z); $
		\item $ ( x)\ \tx{Comp}(s(x),x,x)\wedge \tx{Comp}(x,t(x),x); $
		\item $ ( x,y,z,a,b,c)\ \tx{Comp}(x,y,a)\wedge\tx{Comp}(y,z,b)\wedge\tx{Comp}(a,z,c) \vdash \tx{Comp}(x,b,c). $
	\end{enumerate}	
	Given an $\LL$-structure $A$, axiom (1) is saying that $s_A,t_A\colon A\to A$ factor through the inclusion $\tx{Obj}_A\rightarrowtail A$ to give maps $s'_A,t'_A\colon A\to \tx{Obj}_A$. Then, axiom (2) and (3) together show that $\tx{Comp}_A$ is isomorphic to the pullback of $s'_A$ along $t'_A$. Axioms (4) and (5) express the identity and associativity rules respectively. It is now easy to conclude that $\Mod(\TT)=\Cat(\cv)$ is the $\cv$-category of categories internal to $\cv$.
\end{exam}

\begin{exam}
	Consider $\cv=\Cat$ and take the theory $\TT$ for categories with limits of shape $C\in\Cat$ as in \cite[Example~4.18]{RTe2}. The theory $\TT$ can be modified into a limit theory $\TT'$ by replacing $\psi(x,z)$ with
	$$
	\psi'(x,z):=  (\exists!! w:2*C)\ ( j_0(w)\eqcirc x \wedge j_1(w)\eqcirc z)
	$$
	and the axiom $\beta$ with 
	$$ \beta':= (\forall y:C) (\exists!! x:0*C)\ R(x)\wedge (k(x)\eqcirc y). $$
	A model of $\TT'$ is a category with unique representatives of $C$-limits (that is, any diagram of shape $C$ has a unique limiting cone).
\end{exam}

We now proceed with the theory. Recall that, given a morphism $h:A\to B$ in a $\cv$-category $\ck$, an object $K$ is called $h$-orthogonal if the map
$$ \ck(h,K):\ck(B,K)\to\ck(A,K) $$
is an isomorphism (\cite{K}). Clearly, if $K$ is $h$-orthogonal, it is $h$-injective. 

\begin{defi}
	Given a morphism $h:A\to B$ between $\lambda$-presentable $\mathbb L$-structure $A$ and $B$, we say that an $X$-ary sequent $$\tau_h(x):= \varphi(x) \vdash (\exists!!y)(\psi(x,y)\wedge\varphi(x)),$$ with $\varphi$ and $\psi$ conjunctions of atomic formulas in $\mathbb L_{\lambda\lambda}$, is an \textit{orthogonality sequent} of $h$ if an $\mathbb L$-structure $K$ is $h$-orthogonal if and only if $K\models\tau_h$. 
\end{defi}

\begin{propo}\label{ort}
	Under Assumption~\ref{assumption}, every morphism between $\lambda$-presentable $\mathbb L$-structures has an orthogonality sequent.
\end{propo}
\begin{proof}
	Let $h\colon A\to B$ be a morphism between $\lambda$-presentable $\mathbb L$-structures; then we can consider $\pi^A(x)$, $\pi^B(y)$, and $\chi'(x,y)= (\pi^A(x)\wedge\pi^B(y) \wedge (\tau(y)=\eps(x)))$ as in Lemma~\cite[5.11]{RTe2}. We shall prove that the sequent 
	$$ \tau_h:= \pi^A(x)\vdash (\exists!! y)(\chi'(x,y)) $$
	is an orthogonality sequent for $h$. 
	
	Let $K$ be any $\mathbb L$-structure; we can rewrite part of the diagram from the proof of \cite[5.11]{RTe2} as the solid diagram below.
	\begin{center}
		\begin{tikzpicture}[baseline=(current  bounding  box.south), scale=2]
			
			\node (a0) at (-0.3,2.1) {$\Str(\mathbb L)(B,K)\cong\pi^B_K$};
			\node (c0) at (-0.3,-0.8) {$\Str(\mathbb L)(A,K)\cong\pi^A_K$};
			
			\node (b0) at (3.6,2.1) {$K^{Z_B}$};
			\node (d0) at (3.6,0.9) {$K^{X}$};
			\node (e0) at (2.5,-0.1) {$K^{Z_A}$};
			
			\node (f0) at (2.5,1.1) {$K^{Z_A+Z_B}$};
			\node (g0) at (0.8,1.1) {$\chi'_K$};
			
			\node (l0) at (1.5,0.15) {(a)};

			\path[font=\scriptsize]
			
			(a0) edge [>->] node [above] {} (b0)
			(a0) edge [bend right,->] node [left] {$\pi^h_K:=\Str(\mathbb L)(h,K)$} (c0)
			(b0) edge [->] node [right] {$\tau_K$} (d0)
			(e0) edge [->] node [below] {$\ \ \ \ \ \ \ \quad K^e=\eps_K$} (d0)
			(c0) edge [bend right=20, >->] node [above] {$m$} (e0)
			
			(f0) edge [->] node [above] {} (b0)
			(f0) edge [->] node [above] {} (e0)
			
			(g0) edge [>->] node [above] {} (f0)
			(g0) edge [->] node [left] {$s\ \ $} (a0)
			(g0) edge [] node [right] {$\ \cong$} (a0)

			(g0) edge [bend left,dashed, <-] node [right] {$t$} (c0)
			(g0) edge [->] node [left] {$g'$} (c0);
		\end{tikzpicture}	
	\end{center}  
	where $g'\colon (\pi^B\wedge (\tau\eqcirc \epsilon))_K\to \pi^A_K$ exists by the lemma. Then $K$ is orthogonal with respect to $h$ if and only if $\Str(\mathbb L)(h,K)$ is an isomorphism, if and only if $g'$ is an isomorphism, if and only if there exists a unique $t\colon \pi^A_K\to \chi'_K$ (namely, the inverse of $g'$) making the square (a) commute, if and only if $K$ satisfies $\tau_h$ (since by \cite[5.11]{RTe2} the square (a) with $t$ in place of $g'$ is a pullback).
\end{proof}

We now compare $\cv$-categories of models of limit theories with orthogonality classes. Recall that a {$\lambda$-orthogonality class} in a $\cv$-category $\ck$ is a full subcategory spanned by the objects orthogonal with respect to a set of morphisms with $\lambda$-presentable domain and codomain.

\begin{theo}\label{orth-limitT}
	Consider the following conditions for a full subcategory $\ck$ of $\Str(\LL)$:\begin{enumerate}
		\item $\ck\cong\Mod(\TT)$ for a limit $\mathbb L_{\lambda\lambda}$-theory $\TT$;
		\item $\ck$ is a $\lambda$-orthogonality class in $\Str(\LL)$.
	\end{enumerate}
	Then $(1)\Rightarrow (2)$ always holds, and $(2)\Rightarrow (1)$ holds under Assumption~\ref{assumption}.
\end{theo}
\begin{proof}
	$(2)\Rightarrow (1)$ follows from Proposition~\ref{ort} above. For $(1)\Rightarrow(2)$ it is enough to prove that given any sequent
	$$ \varphi(x) \vdash (\exists!!y)(\psi(x,y)\wedge\varphi(x)),$$
	where $\varphi$ and $\psi$ are conjunctions of atomic formulas, there exists a morphism $g\colon A\to B$ between $\lambda$-presentable $\mathbb L$-structures for which: $K$ is orthogonal with respect to $g$ if and only if $K\models \varphi\vdash \exists!!(\psi\wedge\varphi)$, for any $K\in\Str(\mathbb L)$.
	
	By \cite[Corollary~5.4]{RTe2} there exists $\lambda$-presentable $\mathbb L$-structures $A$ and $C$ for which $\varphi$ and $\psi$ are presentation formulas for $A$ and $C$ respectively; these come together with maps $e\colon F(X)\to A$ and $e'\colon F(X+Y)\to C$. Consider now the pushout $B$ (from the proof of \cite[7.6]{RTe2})
	\begin{center}
		\begin{tikzpicture}[baseline=(current  bounding  box.south), scale=2]

			\node (a0) at (0.1,0.9) {$FX$};
			\node (b0) at (1.2,0.9) {$A$};
			\node (c0) at (0.1,0) {$F(X+Y)$};
			\node (d0) at (1.2,0) {$C$};
			\node (1) at (2,0.45) {$B$};
			
			\path[font=\scriptsize]

			(a0) edge [->] node [left] {$F(i)$} (c0)
			(a0) edge [->] node [above] {$e$} (b0)
			(c0) edge [->] node [below] {$e'$} (d0)
			(b0) edge [->] node [above] {$\ g$} (1)
			(d0) edge [->] node [below] {$\ h$} (1);
		\end{tikzpicture}	
	\end{center} 
	By homming into an $\mathbb L$-structure $K$ and using the definition of presentation formula, we obtain the pullback diagram below.
	\begin{center}
		\begin{tikzpicture}[baseline=(current  bounding  box.south), scale=2]
			
			\node (0) at (-2.4,0.5) {$\Str(\LL)(B,K)\cong(\psi\wedge\varphi)_K$};
			\node (a0) at (0.1,1) {$\Str(\LL)(A,K)\cong \varphi_K$};
			\node (b0) at (1.9,1) {$K^X$};
			\node (c0) at (0.1,0) {$\Str(\LL)(C,K)\cong\psi_K$};
			\node (d0) at (1.9,0) {$K^{X+Y}$};
			
			\path[font=\scriptsize]
			
			(0) edge [->] node [above] {$\Str(\LL)(g,K)\ \ \ \ \ \ \ \ \ $} (a0)
			(0) edge [>->] node [below] {$\Str(\LL)(h,K)\ \ \ \ \ \ \ \ \ $} (c0)
			(a0) edge [>->] node [above] {} (b0)
			(c0) edge [>->] node [below] {} (d0)
			(d0) edge [->] node [right] {$K^i$} (b0)
			(a0) edge [dashed,->] node [right] {$t$} (c0);
		\end{tikzpicture}	
	\end{center} 
	Now, $K$ is orthogonal with respect to $g$ if and only if $\Str(\LL)(g,K)$ is an isomorphism, if and only if $\varphi_K\cong \Str(\LL)(B,K)$ is itself the pullback of the cospan 
	$$\varphi_K\rightarrow K^X\leftarrow \psi_K,$$ 
	if and only if $K$ satisfies the limit sequent.
\end{proof}

\begin{coro}
	For any limit $\mathbb L_{\lambda\lambda}$-theory $\TT$, the $\cv$-category $\Mod(\TT)$ is locally $\lambda$-presentable.
\end{coro}
\begin{proof}
	Follows from Theorem~\ref{orth-limitT} above and \cite[Theorem~3.19]{Te}.
\end{proof}

Recall from \cite{RTe2} the following notation that generalizes to this setting the ordinary way of interpreting unique existential quantification.

\begin{nota}
	We say that an $\LL$-structure $A$ satisfies a sequent of the form
	$$
	 \varphi(x) \vdash (\exists !y)\psi(x,y),
	$$
	if it satisfies
	$$
	 \varphi(x) \vdash (\exists y)\psi(x,y),
	$$
	and
	$$
	  \varphi(x) \wedge \psi(x,y)\wedge\psi(x,y')\ \vdash \ y\eqcirc y'.
	$$
\end{nota}

The enriched unique existential quantification $\exists!!$ is stronger than the one just introduced:

\begin{lemma}\label{unique}
	If $A\models (\varphi\vdash \exists!!(\psi\wedge\varphi))$ then also $ A\models (\varphi\vdash \exists!(\psi\wedge\varphi))$.
\end{lemma}
\begin{proof}
	Consider the solid part of the diagram below.
	\begin{center}
		\begin{tikzpicture}[baseline=(current  bounding  box.south), scale=2]
			
			\node (-1) at (-1.1,1.2) {$\psi'_A$};
			\node (a0) at (-0.1,0.6) {$\psi_A$};
			\node (c0) at (2,0.6) {$ A^{X+Y} $};
			\node (c1) at (2,1.8) {$ A^X$};
			\node (d0) at (1,1.2) {$(\exists y)(\psi\wedge\varphi)_A$};
			\node (0) at (-0.1,1.8) {$\varphi_A$};
			
			\path[font=\scriptsize]
			
			(a0) edge [>->] node [below] {} (c0)
			(c0) edge [->] node [right] {$A^i$} (c1)
			(a0) edge [->>] node [below] {$\ q$} (d0)
			(d0) edge [>->] node [above] {} (c1)
			(0) edge [>->] node [below] {} (c1)
			(0) edge [dashed,->] node [left] {$t$} (a0)
			(0) edge [dashed,->] node [above] {$\ s$} (d0)
			(-1) edge [->] node [above] {$m\ $} (0)
			(-1) edge [>->] node [below] {$n\ $} (a0);
		\end{tikzpicture}	
	\end{center}
	where $\psi'_A$ is the pullback of the cospan $\varphi_A\rightarrow A^X\leftarrow \psi_A$. A close inspection on the validity of
	$$ \big(\varphi(x)\wedge\psi(x,y)\wedge\psi(x,y')\big) \vdash (y=y')$$
	says that $A$ satisfies it if and only if $m\colon\psi'_A\to\varphi_A$ is a monomorphism.
	
	Now, if $A$ satisfies the sequent $(\varphi\vdash \exists!!\psi)$ them $m$ is an isomorphism by Lemma~\ref{pullback-validity}; thus we can conclude.
\end{proof}

%\begin{rem}
%Example~\cite[4.18]{RTe2} provides the instance of a $\cv$-category that can be expressed as the $\cv$-category of models of a theory with sequents of the form 
%	$$(\forall x)(  \varphi(x) \vdash (\exists!y)\psi(x,y)),$$
%	but that is not expressible as the $\cv$-category of models of a limit theory. Indeed, Example~\cite[4.18]{RTe2} presents the 2-category of categories with a specified class of limits, as the 2-category of models of a theory with $\exists!$ as displayed above. On the other hand, this 2-category is not complete (it only has flexible limits), so it cannot be expressed as a $\lambda$-orthogonality class, and hence not even by a limit theory (Corollary~\ref{orth-limitT}).
%\end{rem}

In the following situation limit theories can be captured under regular logic:

\begin{propo}
	Assume that $(\ce,\cm)$=(strong epimorphism, monomorphism) and that strong epimorphisms are pullback stable (this is equivalent to $\cv$ being regular). Then an $\LL$-structure $K$ satisfies the limit sequent 
	$$  \varphi(x) \vdash (\exists!! y) (\psi(x,y)\wedge\varphi(x))$$
	if and only if it satisfies $ \varphi(x) \vdash (\exists! y) (\psi(x,y)\wedge\varphi(x))$. That is, if and only if $K$ satisfies
	$$  \varphi(x) \vdash (\exists y)(\psi(x,y)\wedge\varphi(x)),$$
	$$ \big(\varphi(x)\wedge\psi(x,y)\wedge\psi(x,y')\big) \vdash (y\eqcirc y').$$
\end{propo}
\begin{proof}
	Necessity follows from \ref{unique}.
	
	Conversely consider the diagram in the proof of \ref{unique} and assume that $m$ is a monomorphism and that $s\colon\varphi_K\to (\exists y)(\psi\wedge\varphi)_K$ exists. Then $\psi'_K$ is the pullback of $s$ along $q$; so that $m$ is also a strong epimorphism. Thus $m$ is an isomorphism and $K$ satisfies the limit sequent by Lemma~\ref{pullback-validity}.
\end{proof}

\begin{propo}
	Suppose that the assumptions of \cite[5.8]{RTe1} hold. Then every $\lambda$-orthogonality class is a $\lambda$-injectivity class in $\Str(\mathbb L)$. In particular every $\cv$-category of models of a limit theory is the $\cv$-category of models of a regular theory.
\end{propo}
\begin{proof}
	The first part follows from \cite[Corollary~5.8]{RTe1}. The second by~\cite[Theorem~7.6]{RTe2} and Theorem~\ref{orth-limitT} above.
\end{proof}

\end{document}